\documentclass[11pt]{article}
\usepackage{amssymb}
\usepackage{amsmath}
\usepackage{amsthm}
\usepackage{hyperref}
\usepackage{latexsym}
\usepackage{mathdots}
\usepackage{graphicx}
\usepackage[capitalize]{cleveref}
\usepackage[all]{xy}
\usepackage{xcolor}
\reversemarginpar

\allowdisplaybreaks
\theoremstyle{definition}
\newtheorem{theo}{Theorem}[section]
\newtheorem{defi}[theo]{Definition}
\newtheorem{prop}[theo]{Proposition}

\newtheorem{cor}[theo]{Corollary}
\newtheorem{lemma}[theo]{Lemma}
\newtheorem{exa}[theo]{Example}
\newtheorem{rem}[theo]{Remark}
\newtheorem{nota}[theo]{Notation}
\numberwithin{equation}{section}

\newcommand{\N}{{\mathbb N}}
\newcommand{\F}{{\mathbb F}}
\newcommand{\Z}{{\mathbb Z}}

\newcommand{\cA}{{\mathcal A}}
\newcommand{\cB}{{\mathcal B}}
\newcommand{\cC}{{\mathcal C}}

\newcommand{\cF}{{\mathcal F}}

\newcommand{\cI}{{\mathcal I}}
\newcommand{\cL}{{\mathcal L}}
\newcommand{\cM}{{\mathcal M}}
\newcommand{\cN}{{\mathcal N}}
\newcommand{\cO}{{\mathcal O}}
\newcommand{\cP}{{\mathcal P}}
\newcommand{\cQ}{{\mathcal Q}}

\newcommand{\cU}{{\mathcal U}}
\newcommand{\cV}{{\mathcal V}}
\newcommand{\cW}{{\mathcal W}}

\newcommand{\cZ}{{\mathcal Z}}

\newcommand{\qM}{$q$-matroid}
\newcommand{\GL}{\mbox{\rm GL}}
\newcommand{\rk}{\mbox{\rm rk}\,}
\newcommand{\rs}{\mbox{\rm rowsp}}

\newcommand{\Hyp}{\mbox{\rm Hyp}}
\newcommand{\M}{\mbox{\rm M}}
\newcommand{\W}{\mbox{\rm W}}

\newcommand{\Config}{\mbox{\rm Config}}
\newcommand{\CConfig}{\mbox{$\mbox{\rm Config}^*_{\cP}$}}
\newcommand{\CF}{\mbox{\rm CF}}

\newcommand{\cl}{\text{cl}}
\newcommand{\cyc}{\text{cyc}}

\newcommand{\T}{\mbox{$\!^{\sf T}$}}

\newcommand{\subspace}[1]{\mbox{$\langle{#1}\rangle$}}
\newcommand{\inner}[2]{\mbox{$\langle{#1}\!\mid\!{#2}\rangle$}}

\newcommand{\BinomS}[2]{\genfrac{(}{)}{0pt}{2}{#1}{#2}}
\newcommand{\Gaussian}[2]{\genfrac{[}{]}{0pt}{1}{#1}{#2}}
\newcommand{\GaussianD}[2]{\genfrac{[}{]}{0pt}{0}{#1}{#2}}
\newcommand{\CloudU}{\mbox{$\mbox{\rm cloud}^{\uparrow}$}}
\newcommand{\FlockD}{\mbox{$\mbox{\rm flock}^{\downarrow}$}}

\newenvironment{liste}{\begin{list}{--\hfill}{\topsep-1.5ex \labelwidth.4cm
   \leftmargin.5cm \labelsep.1cm \rightmargin0cm \parsep0ex \itemsep.6ex
   \partopsep-1.5ex}}{\end{list}}
\newcounter{alp}
\newcounter{ara}
\newcounter{rom}

\newenvironment{alphalist}{\begin{list}{(\alph{alp})\hfill}{\usecounter{alp}
     \topsep-1.5ex \labelwidth.7cm \leftmargin.7cm \labelsep0cm
     \rightmargin0cm \parsep0ex \itemsep0ex
     \partopsep-1.5ex}}{\end{list}}
\newenvironment{arabiclist}{\begin{list}{(\arabic{ara})\hfill}{\usecounter{ara}
     \topsep-1.5ex \labelwidth.7cm \leftmargin.7cm \labelsep0cm
     \rightmargin0cm \parsep0ex \itemsep0ex
     \partopsep-1.5ex}}{\end{list}}

\begin{document}

\title{The Cloud and Flock Polynomials of $q$-Matroids}
\author{Heide Gluesing-Luerssen\thanks{Department of Mathematics, University of Kentucky, Lexington KY 40506-0027, USA; heide.gl@uky.edu. HGL is the corresponding author.}\quad and Benjamin Jany\thanks{Department of Mathematics and Computer Science, Eindhoven University of Technology, the Netherlands; b.jany@tue.nl. BJ is supported by the Dutch Research Council through grant VI.Vidi.203.045.}
}

\date{April 7, 2025}
\maketitle
	
\begin{abstract}
\noindent
We show that the Whitney function of a \qM{} can be determined from the cloud and flock polynomials associated to the 
cyclic flats. 
These polynomials capture information about the corank (resp., nullity) of certain spaces whose cyclic core (resp., closure) is the given cyclic flat.
Going one step further, we prove that the Whitney function, and in fact the cloud-flock lattice, are determined by the configuration of the \qM{}, which is the abstract lattice of cyclic flats together with the corank-nullity data.
Furthermore, we show that the configuration and cloud-flock lattice behave well under duality and direct sums,
whereas the Whitney function does not contain enough information to behave well under taking direct sums.
As an aside we show that every configuration of a matroid arises as a configuration of a \qM{}, whereas the converse is not true. 
\end{abstract}


\section{Introduction}\label{S-Intro}

In this paper we investigate certain (polynomial) invariants of \qM{}s.  
Since their introduction in 2018 in~\cite{JuPe18}, \qM{}s have attracted a lot of attention due to their close relationship to rank-metric codes~\cite{JuPe18,GJLR19,Shi19,GhJo20,JPV23,GLJ22Gen}.
The theory of \qM{}s, despite its young age, has made substantial progress in the areas of cryptomorphisms~\cite{BCJ22}, 
subspace designs~\cite{BCIJ24}, and structural properties~\cite{Ja23,AlBy24,CeJu24,GLJ24DSCyc}.
In the notoriously hard area of representability some first results have been obtained in \cite{GLJ22Rep,AJNZ24}, but unfortunately even questions with simple answers in matroid theory have proven to be extremely challenging.

It is well-known that for matroids the Whitney function, which is the bivariate generating function of the corank-nullity pairs over all subsets, is related to the Tutte polynomial via a simple variable shift. 
The two polynomials were introduced independently for graphs by Tutte and count different invariants. 
Crapo~\cite{Cra69} generalized them to matroids and established their simple shift relation. 
The Tutte polynomial (and thus the Whitney function) satisfies certain fundamental recursions with respect to matroid minors.
In fact, together with a normalization these recursions define the Tutte polynomial. 

For \qM{}s, the Whitney function can be defined exactly as for matroids.
However, it turns out that it does not satisfy the $q$-analogues of the above-mentioned recursions.
To be more precise, it is not a $q$-Tutte-Grothendieck invariant in the sense of \cite{ByFu25}. 
For this reason, in \cite{ByFu25} the authors search for such a $q$-Tutte Grothendieck invariant for \qM{}s.
They succeed in doing so under the assumption that the subspace lattice of the \qM{} admits a certain type of partition (which unfortunately is not always the case). 
They further show that there exists a bijection in $\Z[x,y]$ relating the Whitney function and the Tutte polynomial. 
This bijection can be expressed as a $\Z$-linear substitution of the monomials $x^iy^j$, and thus differs substantially 
from the simple shift relation in the matroid case. 

In this paper we focus on the Whitney function and introduce further invariants for \qM{}s. 
A prominent role will be played by the collection of cyclic flats.
It is well known \cite{AlBy24,GLJ24DSCyc} that the collection $\cZ(\cM)$ of cyclic flats of a given \qM{}~$\cM$ together with the rank values of the cyclic flats fully determine~$\cM$.
Following the study of matroids in \cite{PlBae14,Ebe14}, we introduce the cloud and flock polynomials associated with the cyclic flats of a \qM{}.
The cloud polynomial associated with $Z\in\cZ(\cM)$ is the generating function for the corank of all flats whose cyclic core is~$Z$, while the flock polynomial is the generating function for the nullity of all the spaces whose closure is~$Z$.
Making use of the powerful properties of~$\cZ(\cM)$, we show
in \cref{C-CloudFlock2} that the collection of all cloud and flock polynomials determines the Whitney function.
This provides us with a $q$-analogue of the corresponding result for matroids \cite{PlBae14,Ebe14}, but it takes a more complicated form than for matroids.
This has the consequence that, in general, the Whitney function of a direct sum $\cM_1\oplus\cM_2$ is not determined by the Whitney functions of 
the summands~$\cM_1$ and~$\cM_2$; see \cref{E-WhitDS}.
This stands in striking contrast to the simple product formula for direct sums of matroids.
Moreover -- as for matroids -- the Whitney function does not determine the cloud and flock polynomials, and actually 
it does not even determine the number of cyclic flats. 

In addition, we introduce the configuration and cloud-flock lattice of a \qM{} (see \cite{Ebe14} for configurations of matroids).
The former is defined by replacing in the lattice of cyclic flats each actual cyclic flat~$Z$ by its corank-nullity pair, while in the latter we replace~$Z$ by its pair of 
cloud and flock polynomial.
Obviously, the configuration contains much less information than the lattice of cyclic flats and, unsurprisingly,
there exist non-equivalent \qM{}s with the same configuration.
Since the degrees of the cloud and flock polynomial associated with~$Z$ are the corank and nullity of~$Z$, the cloud-flock lattice trivially 
determines the configuration.
It turns out that the converse is true as well: the configuration fully determines the cloud-flock lattice and thus the Whitney function.
This is worked out in two steps.
First, in \cref{T-ConfigWhitney} we establish the result for \emph{full} \qM{}s (the $q$-analogue of loopless and coloopless matroids).
Thereafter, we make use of the fact that a non-full \qM{} can be decomposed into the direct sum of a full, a trivial, and a free \qM{}~\cite{GLJ24DSCyc}.
With some lengthy, yet straightforward computations about the relevant invariants of direct sums involving a trivial or free summand, we are then able to prove the general case in \cref{T-ConfigWhitneyGen}.

Just as for matroids, the converse of the just mentioned result is not true, that is, the Whitney function does not determine the configuration -- simply because it does not even determine the number of cyclic flats.
However, as we show in \cref{T-Extremal}, some partial information about the cyclic flats can be retrieved from the Whitney function.

In Corollaries~\ref{C-DualConfig} and \ref{C-DualCF} we show that the configuration and the cloud-flock lattice of a \qM{}~$\cM$ fully determines the respective information for 
the dual~$\cM^*$. However, the duality relation is not as simply as the well-known identity for the Whitney function.
We discuss this at the end of \cref{S-Config}.
In \cref{T-DirSumCF} we establish that the configurations and the cloud-flock lattices of~$\cM_1$ and~$\cM_2$ determine the same 
information for $\cM_1\oplus\cM_2$.

In \cref{S-ConfigMatqMat} we take a brief excursion and show that every finite lattice arises as the lattice of cyclic flats of a \qM{}. 
The proof will make use of the same result for matroids. This will allow us to show that every configuration of a matroid arises as the configuration of a \qM{}, while the converse is not true.

Finally, in \cref{S-Condensed} we introduce condensed configurations. 
They consist of data pertaining to specific partitions of the cyclic flats.
These partitions and their data capture information about the corank-nullity pairs of the cyclic flats and their containment relations.
This information is not sufficient to reconstruct the configuration.
But it is sufficient to fully determine the Whitney function; see \cref{T-CondensWhitney}.
The converse is not true.
This section closely follows the matroid case developed in~\cite{Ebe14}.

\textbf{Notation:} Throughout, we let $\F=\F_q$.
For any finite-dimensional $\F$-vector space~$E$ we define $\cL(E)$ as the subspace lattice of~$E$.
For a matrix $M\in\F^{k\times n}$ we use $\rs(M)\subseteq\F^n$ for the row space of~$M$.
Furthermore, we denote by $\Hyp(V)$ the set of all hyperplanes of~$V$.
These are the codimension-1 subspaces of~$V$ (and must not be confused with hyperplanes in the ($q$-)matroid sense).
We set $[n]=\{1,\ldots,n\}$ and $[n]_0=\{0,1,\ldots,n\}$.
The abbreviation RREF stands for reduced row echelon form.

\section{Basic Notions of $q$-Matroids}

In this section we collect some basic notions and facts for \qM{}s.

\begin{defi}\label{D-qMatroid}
A \textbf{$q$-matroid with ground space~$E$} is a pair $\cM=(E,\rho)$, where~$E$ is a finite-dimensional $\F$-vector space
and $\rho: \cL(E)\longrightarrow\N_{\geq0}$ is a map satisfying
\begin{liste}
\item[(R1)\hfill] Dimension-Boundedness: $0\leq\rho(V)\leq \dim V$  for all $V\in\cL(E)$;
\item[(R2)\hfill] Monotonicity: $V\leq W\Longrightarrow \rho(V)\leq \rho(W)$  for all $V,W\in\cL(E)$;
\item[(R3)\hfill] Submodularity: $\rho(V+W)+\rho(V\cap W)\leq \rho(V)+\rho(W)$ for all $V,W\in\cL(E)$.
\end{liste}
We call $\rho(V)$ the \textbf{rank of}~$V$ and $\rho(\cM):=\rho(E)$ the rank of the \qM.
\end{defi}

\begin{defi}\label{D-Equiv}
Two \qM{}s $\cM_i=(E_i,\rho_i)$ are \textbf{equivalent}, denoted by $\cM_1\approx\cM_2$, if there exists an $\F$-isomorphism $\alpha:E_1\longrightarrow E_2$ such that $\rho_2(\alpha(V))=\rho_1(V)$ for all $V\in\cL(E_1)$.
\end{defi}

The following notions are crucial in any study of \qM{}s.

\begin{defi}\label{D-FlatsOpenCyclic}
Let $\cM=(E,\rho)$ be a \qM. 
A subspace $V\in\cL(E)$ is \textbf{independent} if $\rho(V)=\dim V$ and \textbf{dependent} otherwise.
A dependent space all of whose proper subspaces are independent is called a \textbf{circuit}.
The subspace~$V$ is a \textbf{flat} if it is inclusion-maximal in the set $\{W\in\cL(E)\mid \rho(W)=\rho(V)\}$, in other words $\rho(V)<\rho(V+\subspace{x})$ for all $x\in E\setminus V$. 
Moreover,~$V$ is \textbf{cyclic} (or \textbf{open}) if $\rho(W)=\rho(V)$ for all $W\in\Hyp(V)$. 
We set $\cF(\cM)=\{F\in\cL(E)\mid F\text{ flat in }\cM\}$ and 
$\cO(\cM)=\{O\in\cL(E)\mid O\text{ cyclic in }\cM\}$ and 
\[
    \cZ(\cM)=\cF(\cM)\cap\cO(\cM)=\{Z\in\cL(E)\mid Z\text{ cyclic flat in }\cM\}.
\]
We call~$\cM$ \textbf{full} if $0$ and $E$ are cyclic flats.
\end{defi}

Note that a full \qM{} is the $q$-analogue of a loopless and coloopless matroid; 
see also \cite[Rem.~7.8]{GLJ24DSCyc} for a more detailed discussion.

It is well known that the collection of independent spaces uniquely determines the $q$-matroid.
The same is true for the flats, dependent spaces, open spaces, and circuits.
We refer to \cite[Thm.~8]{JuPe18} and \cite{BCJ22} for the according cryptomorphisms.
Moreover, the cyclic flats together with their rank values also uniquely determine the $q$-matroid.
Indeed, 
\begin{equation}\label{e-RhoqM}
   \rho(V)=\min_{Z\in\cZ(\cM)}\Big(\rho(Z)+\dim((V+Z)/Z)\Big) \ \text{ for all }V\in\cL(E);
\end{equation}
see \cite[Cor.~4.6]{GLJ24DSCyc}. 
The resulting cryptomorphism can be found in \cite[Cor.~4.12]{AlBy24}.

A large class of $q$-matroids are the representable ones.

\begin{theo}[\mbox{\cite[Sec.~5]{JuPe18}}]\label{T-ReprqMatr}
Let $\F_{q^m}$ be a field extension of $\F=\F_q$ and let $G\in\F_{q^m}^{k\times n}$.
Define the map $\rho:\cL(\F^n)\longrightarrow\N_0$ via
\[
    \rho(\rs(Y))=\rk(GY\T)\ \text{ for all matrices $Y\in\F^{t\times n}$ and all $t\in\{0,\ldots,n\}$}
\]
(where the rank is over $\F_{q^m}$).
Then~$\rho$ is well-defined and $\cM_G:=(\F^n,\rho)$ is a $q$-matroid. It is called the $q$-matroid \textbf{represented by $G$}.
A $q$-matroid $\cM$ of rank~$k$ with ground space~$E$ is called \textbf{$\F_{q^m}$-representable} if $\cM\approx\cM_G$ for some matrix~$G\in\F_{q^m}^{k\times n}$.
\end{theo}

The row space of~$G$ is known as the rank-metric code generated by~$G$.
Many invariants of the code, but not all, can be determined by the associated $q$-matroid.
We refer to \cite{JuPe18,GJLR19,BCIJ24,GLJ22Gen} and Section~6 in the survey \cite{Gor21} for further details.

\begin{exa}\label{E-UnifRepr}
Let $k\in[n]_0$. 
The \textbf{uniform \qM{} of rank~$k$} with $n$-dimensional ground space~$E$ is defined to be $(E,\rho)$, where  $\rho(V)=\min\{k,\dim V\}$ for all 
$V\in\cL(E)$. 
It is denoted by $\cU_{k,n}(E)$ or simply $\cU_{k,n}$ (in which case the underlying field~$\F_q$ needs to be clear from the context).
Note that $\cU_{0,n}$ is represented by the $1\times n$-zero matrix and 
$\cU_{n,n}$ by the $n\times n$-identity matrix.
For $0<k<n$, the uniform \qM{} $\cU_{k,n}$ is representable over $\F_{q^m}$ if and only if $m\geq n$; see \cite[Rem.~3.10]{Gor21}.
We call $\cU_{0,n}$ and $\cU_{n,n}$ the \textbf{trivial} and the \textbf{free} \qM{}, respectively.
\end{exa}

Just like for matroids, not every $q$-matroid is representable. Examples can be found in \cite[Sec.~4]{GLJ22Gen} and \cite[Sec.~3.3]{CeJu24}. See also \cref{T-DSProperties}(b) below.

Restriction and contraction of \qM{}s are defined in the usual way.

\begin{defi}\label{D-RestrContr}
Let $\cM=(E,\rho)$ be a \qM{} and $X\in\cL(E)$.
The \textbf{restriction of~$\cM$ to~$X$} and the \textbf{contraction of~$X$ from~$\cM$}
are defined as the \qM{}s $\cM|X=(X,\hat{\rho})$ with $\hat{\rho}(V)=\rho(V)$ for all $V\leq X$ and
$\cM/X=(E/X,\tilde{\rho})$ with $\tilde{\rho}(V/X)=\rho(V)-\rho(X)$ for all $V\in\cL(E)$ containing~$X$, respectively.
\end{defi}

Occasionally we will make use of the dual \qM{}.
The following can be found for instance in \cite[Def.~41, Thm.~42]{JuPe18}, where duality is based on a lattice involution, 
or in \cite[Thm.~2.8]{GLJ22Gen} based on a non-degenerate symmetric bilinear form on~$E$. 
In~\cite{GLJ22Gen} it is also explained that, while~$\cM^*$ depends on the choice of the bilinear form, it is 
unique up to equivalence.

\begin{theo}\label{T-DualqM}
Let $\inner{\cdot}{\cdot}$ be a non-degenerate symmetric bilinear form (NSBF) on~$E$, and for $V\in\cL(E)$ set
$V^\perp=\{w\in E\mid \inner{v}{w}=0\text{ for all }v\in V\}$.
Given a \qM{} $\cM=(E,\rho)$. Then
$\rho^*(V)=\dim V+\rho(V^\perp)-\rho(E)$
defines a rank function and thus $\cM^*=(E,\rho^*)$ is a \qM{}, called the \textbf{dual} of~$\cM$ with respect 
to the chosen NSBF.
Furthermore, $\cM^{**}=\cM$, where $\cM^{**}=(\cM^*)^*$ is the bidual.
\end{theo}

\begin{rem}\label{R-FlatsOpen}
For all $V\in\cL(E)$ one has $V\in\cF(\cM)\Longleftrightarrow V^\perp\in\cO(\cM^*)$; see \cite[Cor.~86]{BCJ22}. 
As a consequence, $\cZ(\cM^*)=\{Z^\perp\mid Z\in\cZ(\cM)\}$.
\end{rem}

We now turn to the direct sum of $q$-matroids, which has been introduced in \cite{CeJu24}.
Throughout the paper we use the following notation.

\begin{nota}\label{Nota}
For any direct sum $E=E_1\oplus E_2$ of $\F$-vector spaces~$E_1$ and~$E_2$ we denote by
$\pi_i:E\longrightarrow E_i,\,i=1,2,$  the corresponding projections.
Moreover, for any collections $\cV_i$ of subspaces in~$E_i$ we set 
$\cV_1\oplus\cV_2=\{V_1\oplus V_2\mid V_i\in\cV_i\}$.
\end{nota}

\begin{theo}[\mbox{\cite[Sec.~6]{CeJu24} and \cite[Ch.~5]{GLJ23C}}]\label{T-DirSum}
Let $\cM_i=(E_i,\rho_i),\,i=1,2,$ be \qM{}s and set $E=E_1\oplus E_2$.
Define $\rho'_i:\cL(E)\longrightarrow \N_0,\ V\longmapsto \rho_i(\pi_i(V))$ for $i=1,2$.
Then $\cM'_i=(E,\rho'_i)$ is a \qM{}.
Define
\begin{equation}\label{e-rho}
  \rho:\cL(E)\longrightarrow\N_0,\quad V\longmapsto\dim V+\min_{X\in\cL(V)}\big(\rho'_1(X)+\rho'_2(X)-\dim X\big).
\end{equation}
Then $\cM:=(E,\rho)$ is a $q$-matroid, called the \textbf{direct sum} of $\cM_1$ and~$\cM_2$ and denoted by $\cM_1\oplus\cM_2$.
\end{theo}

It is easy to see that 
\begin{equation}\label{e-RhoDirect}
  \rho(V_1\oplus V_2)=\rho_1(V_1)+\rho_2(V_2)\ \text{ for all $V_i\in\cL(E_i),\,i=1,2$}.
\end{equation}
In particular $\rho(\cM)=\rho_1(\cM_1)+\rho_2(\cM_2)$.

\begin{rem}\label{R-MiPrime}
The auxiliary \qM{}s $\cM_i'$ appearing in \cref{T-DirSum} 
are in fact special cases of the direct sum. 
Indeed, $\cM_1'\approx\cM_1\oplus\cU_{0,n_2}$ and $\cM_2'\approx\cU_{0,n_1}\oplus\cM_2$.
\end{rem}

The direct sum of \qM{}s behaves very different from that for matroids. 
Some discrepancies are summarized in the following result.
We will see more differences later in this paper.
Recall \cref{Nota}.

\begin{theo}\label{T-DSProperties}
Let $\cM_i=(E_i,\rho_i),\,i=1,2,$ and $\cM=\cM_1\oplus\cM_2$.
\begin{alphalist}
\item Similar to \cref{D-FlatsOpenCyclic}, for a \qM{}~$\cN$ let $\cI(\cN)$ and $\cC(\cN)$ denote the collection of independent spaces and circuits, respectively. 
    Then $\cI(\cM_1)\oplus\cI(\cM_2)\subset\cI(\cM),\  \cF(\cM_1)\oplus\cF(\cM_2)\subset\cF(\cM),\  \cO(\cM_1)\oplus\cO(\cM_2)\subset\cO(\cM)$,
    and $\cC(\cM_1)\cup\cC(\cM_2)\subset\cC(\cM)$. In general equality does not hold in any of these cases.
    see \cite[Prop.~6.1]{GLJ24DSCyc}. This stands in contrasts to the matroid case where one has equality for all collections.
\item $\cM$ may be non-representable over any field extension of~$\F$ even if $\cM_1$ and $\cM_2$ are $\F_{q^m}$-representable;
     see \cite{GLJ22Rep} and \cite{AJNZ24} for further details.
     Again, this differs strikingly from the matroid case.
\end{alphalist}
\end{theo}

On the positive side, the cyclic flats behave well with the direct sum.

\begin{theo}[\mbox{\cite[Thms.~5.6 and 6.3]{GLJ24DSCyc}}]\label{T-DSCycFlats}
Let $\cM_i=(E_i,\rho_i),\,i=1,2,$ and $\cM=\cM_1\oplus\cM_2$ with rank function~$\rho$.
Then $\cZ(\cM_1)\oplus\cZ(\cM_2)=\cZ(\cM)$. As a consequence (see \eqref{e-RhoqM}),
\[
 \rho(V)=\dim V+\min_{Z_i\in\cZ(\cM_i)}\big(\rho_1(Z_1)+\rho_2(Z_2)-\dim((Z_1\oplus Z_2)\cap V)\big)
\]
for all $V\in\cL(E_1\oplus E_2)$.
\end{theo}

\section{The Whitney Function of $q$-Matroids}\label{S-Whitney}

Throughout this section, let $\F=\F_q$ and $E$ be an $n$-dimensional $\F$-vector space. 
Fix a \qM{} $\cM=(E,\rho)$ and let $\hat{\rho}=\rho(E)$ be its rank.

The Whitney function of~$\cM$ (also known as rank-generating function or corank-nullity function) is defined in the usual way; see also  
\cite[Def.~64]{ByFu25}.

\begin{defi}\label{D-Polys}
The \textbf{Whitney function} or \textbf{corank-nullity function} of~$\cM$ is 
\[
   R_\cM=\sum_{V\leq E}x^{\hat{\rho}-\rho(V)}y^{\dim V -\rho(V)}\in\Z[x,y].
\]
The exponent $\hat{\rho}-\rho(V)$ is called the \textbf{corank} of~$V$ and $\dim V -\rho(V)$ is the \textbf{nullity}.
\end{defi}

Clearly, the Whitney function is invariant under equivalence of \qM{}s.

The coefficients of the Whitney function can easily be described.
A monomial $x^iy^j$ appears in~$R_\cM$ if and only if there exists a subspace $V\in\cL(E)$ such that $\rho(V)=\hat{\rho}-i$ and 
$\dim V=\hat{\rho}-i+j$.
Furthermore, in that case $j\leq n-\hat{\rho}$. Indeed, 
choosing a subspace~$W$ such that $V\oplus W=E$, we obtain from submodularity
$\hat{\rho}=\rho(E)\leq\rho(V)+\rho(W)\leq\rho(V)+\dim W=\hat{\rho}-i+n-\dim V=n-j$ and thus $j\leq n-\hat{\rho}$.
All of this shows that the Whitney function is of the form
\begin{equation}\label{e-WhitExpand}
   R_{\cM}=\sum_{i=0}^{\hat{\rho}} \sum_{j=0}^{n-\hat{\rho}} \nu_{i,j}x^{i}y^{j}
\end{equation}
where 
\begin{equation}\label{e-WhitCoeff}
   \nu_{i,j}= \big|\{V\leq E\mid \rho(V)=\hat{\rho}-i,\,\dim V=\hat{\rho}-i+j\}\big|.      
\end{equation}

\begin{rem}\label{R-WhitneyCounting}
The Whitney function determines various parameters of~$\cM$. First of all, we have
\[
   \hat{\rho}=\deg_x R_{\cM},\quad n-\hat{\rho}=\deg_y R_{\cM},\quad 
     |\cL(E)|=R_\cM(1,1),
\]
and $x^{\hat{\rho}}$ and $y^{n-\hat{\rho}}$ are monomials in~$R_\cM$.
Furthermore, $\sum_{i=0}^k\nu_{\hat{\rho}-i,k-i}=\Gaussian{n}{k}$ for every $k\in[n]$ and 
\begin{align*}
   R_\cM(0,0)&=\text{number of bases of }\cM,\\
   R_\cM(1,0)&=\text{number of independent spaces of }\cM,\\
   R_\cM(0,1)&=\text{number of spanning spaces of }\cM
\end{align*}
(see \cite{BCJ22} for bases and spanning spaces).
Finally, if $R_\cM=\sum_{i=0}^{\hat{\rho}} f_ix^i=\sum_{j=0}^{n-\hat{\rho}}g_j y^j$ where $f_i\in\Z[y]$ and $g_j\in\Z[x]$, then 
\[
     0\in\cF(\cM)\Longleftrightarrow f_{\hat{\rho}}=1\ \text{ and }\ E\in\cO(\cM)\Longleftrightarrow g_{n-\hat{\rho}}=1.
\]
Indeed, $0\not\in\cF(\cM)$ is equivalent to the existence of a space $V\in\cL(E)\setminus\{0\}\text{ such that } \rho(V)=0$, and this in turn 
is equivalent to $x^{\hat{\rho}-\rho(V)}y^{\dim V-\rho(V)}=x^{\hat{\rho}}y^\ell$ being a monomial in~$R_\cM$ for some $\ell>0$.
 Likewise, $E\not\in\cO(\cM)$ is equivalent to the existence of some $W\in\Hyp(E)$ such that $\rho(W)=\hat{\rho}-1$, and this is equivalent to
 $xy^{n-\hat{\rho}}$ being a monomial in~$R_\cM$.
 Later in \cref{T-Extremal} we will see how to extract partial information about the cyclic flats of~$\cM$ from the Whitney function.
\end{rem}

The following is easy to see.

\begin{exa}\label{E-WhitneyUnif}
Let $k\in[n]_0$ and $\cM=\cU_{k,n}$. Then $R_\cM=\sum_{j=0}^k\Gaussian{n}{j}x^{k-j}+\sum_{j=k+1}^n\Gaussian{n}{j}y^{j-k}$.
\end{exa}

Just like for matroids, there exist non-equivalent \qM{}s with the same Whitney function.
The following example is not hard to find. We will see more interesting cases later in 
Examples~\ref{E-CloudFlock3},~\ref{E-CConfig}, and \ref{E-DiffCycFlats}.

\begin{exa}\label{E-NonIsoWhitney}
Let $\F=\F_2$. In $E=\F^6$ define
\[
   V_1=\subspace{e_1,e_2,e_3},\ V_2=\subspace{e_4,e_5,e_6},\ V_3=\subspace{e_3,e_4,e_5}
\]
and set $\cV=\{V_1,\,V_2\}$ and $\cW=\{V_1,\,V_3\}$.
Define $\rho_i:\cL(E)\longrightarrow\N_0,\,i=1,2,$ via
\[
  \rho_1(V)=\left\{\begin{array}{cl}2,&\text{if }V\in\cV,\\ \min\{\dim V,3\},&\text{otherwise,}\end{array}\right\}\qquad
  \rho_2(V)=\left\{\begin{array}{cl}2,&\text{if }V\in\cW,\\ \min\{\dim V,3\},&\text{otherwise.}\end{array}\right\}
\]
Then~$\rho_1$ and~$\rho_2$ are rank functions (see \cite[Prop.~4.6]{GLJ22Gen}) and thus we have \qM{}s $\cM_i=(E,\rho_i)$. 
It turns out that 
\[
   R_{\cM_1}=R_{\cM_2}=x^3 + y^3 + 63x^2 + 2xy + 63y^2 + 651x + 651y + 1393.
\]
The two \qM{}s are not equivalent.
Indeed, the lattices of cyclic flats are given by $\cZ(\cM_1)=\cV\cup\{0,\F^6\}$ and $\cZ(\cM_2)=\cW\cup\{0,\F^6\}$,
and $V_1\cap V_2=0$, whereas $V_1\cap V_3\neq0$.
Note that the abstract lattices of cyclic flats are isomorphic, and the isomorphism preserves the dimension and 
rank values of the cyclic flats.
Later in \cref{S-Config} we will describe this via the configurations of  $\cM_1$ and $\cM_2$.
\end{exa}

Recall the dual \qM{} $\cM^*$ from \cref{T-DualqM}. 
Obviously, the rank functions~$\rho$ and~$\rho^*$ satisfy
\begin{equation}\label{e-coranknullityduality}
    \dim V^\perp-\rho^*(V^\perp)=\rho(E)-\rho(V)\ \text{ for all }V\in\cL(E),
\end{equation}
that is, the nullity of~$V^\perp$ in~$\cM^*$ equals the corank of~$V$ in~$\cM$.
This leads immediately to the following duality of the Whitney function; see also \cite[Prop.~92]{ByFu25}.

\begin{theo}\label{T-WhitDual}
$R_{\cM^*}(x,y)=R_{\cM}(y,x)$.
\end{theo}

We close this section with a short remark on the characteristic function.

\begin{rem}\label{R-CharPoly}
The characteristic polynomial of~$\cM$ is defined as 
$\chi_\cM=\sum_{V\leq E}\mu(0,V)x^{\hat{\rho}-\rho(V)}\in\Z[x]$, where $\mu$ is the M{\"o}bius function on the subspace lattice $\cL(E)$ (see also
\cite[Def.~23]{BCIJ24}). 
Thus, $\chi_{\cM}=\sum_{j=0}^n\sum_{\dim V=j}(-1)^jq^{\BinomS{j}{2}}x^{\hat{\rho}-\rho(V)}$.
With the notation from \eqref{e-WhitExpand} and \eqref{e-WhitCoeff} we have 
\[
    \chi_{\cM}=\sum_{i=0}^{\hat{\rho}}\sum_{j=i}^{n}(-1)^j\nu_{\hat{\rho}-i,j-i}q^{\BinomS{j}{2}}x^{\hat{\rho}-i}
    =\sum_{i=0}^{\hat{\rho}}\sum_{j=0}^{n-\hat{\rho}}\nu_{i,j}(-1)^{\hat{\rho}+j-i}q^{\BinomS{\hat{\rho}+j-i}{2}}x^i.
 \]
In other words, $\chi_\cM$ is obtained from~$R_\cM$ by substituting $(-1)^{\hat{\rho}+j-i}q^{\BinomS{\hat{\rho}+j-i}{2}}x^i$ for $x^iy^j$.
This can be rephrased as substituting  $q^{\BinomS{\hat{\rho}+j-i}{2}}(-x)^i(-1)^j$ for $x^iy^j$ in the polynomial $(-1)^{\hat{\rho}}R_\cM$.
The latter may be regarded as the $q$-analogue of the well-known matroid identity $\chi_M(x)=(-1)^{\hat{r}}R_M(-x,-1)$, where $\hat{r}$ is the rank of the matroid~$M$.
\end{rem}

\section{The Cloud and Flock Polynomial}\label{S-CloudFlock}

Following \cite{PlBae14,Ebe14} on clouds and flocks for matroids, we introduce the cloud and flock polynomials of a \qM.
The cloud polynomial of a cyclic flat~$Z$ captures information about the corank of the flats whose cyclic core is~$Z$, while the flock polynomial of~$Z$ contains information about the nullity of the spaces whose closure is~$Z$. 
The main result of this section shows that the collection of all cloud and flock polynomials of~$\cM$ fully determine the Whitney function~$R_\cM$.

Throughout, let $\cM=(E,\rho)$ be a \qM{} with $\rho(\cM)=\hat{\rho}$. Recall the notation from \cref{D-FlatsOpenCyclic}.
Closely related to the notion of flats and cyclic spaces are the closure and cyclic core.
For $V\in\cL(E)$ we define the \textbf{closure} of~$V$ as
\[
   \cl(V)=\sum_{   \genfrac{}{}{0pt}{1}{x\in E}{\rho(V+\langle x\rangle)=\rho(V)}     }\hspace*{-1.7em}\subspace{x},
\]
and the \textbf{cyclic core} as
\[
   \cyc(V)=\{x\in V\mid \rho(W)=\rho(V)\text{ for all $W\leq V$ such that $W+\subspace{x}=V$}\}.
\]
Clearly, $\cl(V)\in\cF(\cM)$ and $\cyc(V)\in\cO(\cM)$ for all $V\in\cL(E)$; see also \cite[Thm.~3.6]{GLJ24DSCyc}.
Properties of the closure can be found in \cite[Def.~15]{BCJ22} (see also the summary in \cite[Thm.~2.7]{GLJ24DSCyc}), while the cyclic core has been studied in detail in \cite{AlBy24} and \cite{GLJ24DSCyc}.
We will need the following results, which can be found in 
\cite[Prop.~3.10(a), Cor.~3.7, Lem.~4.1]{GLJ24DSCyc} or \cite[Lemmas~3.17, 3.23, 3.22]{AlBy24}.

\begin{rem}\label{R-cyccl}
Let $V\in\cL(E)$. Then 
\begin{alphalist}
\item $\dim V-\rho(V)=\dim\cyc(V)-\rho(\cyc(V))$;
\item $\cl^*(V^\perp)=\cyc(V)^\perp$, where $\cl^*$ is the closure in~$\cM^*$;
\item $F\in\cF(\cM)\Longrightarrow \cyc(F)\in\cZ(\cM)$.
\end{alphalist}
\end{rem}

Thus we obtain two well-defined operators
\begin{equation}\label{e-cyccl}
\left.\begin{array}{ccccrcc}
  \cl:&\;\cL(E)&\longrightarrow&\cF(\cM),\ &\quad V&\longmapsto& \cl(V),\\[.5ex]
  \cyc:&\;\cF(\cM)&\longrightarrow&\cZ(\cM),\ &\quad F&\longmapsto& \cyc(F).
\end{array}\qquad\right\}
\end{equation}

The following notions are the $q$-analogues of the cloud and flock for matroids as studied in \cite{PlBae14} and \cite{Ebe14}.
Note that $\cyc^{-1}$ and $\cl^{-1}$ below refer to the operators in~\eqref{e-cyccl}, hence the spaces in $\cyc^{-1}(Z)$ are flats by definition.

\begin{defi}\label{D-CloudFlock}
Let $Z\in\cZ(\cM)$.
The \textbf{cloud} of~$Z$ and \textbf{flock} of~$Z$ are, respectively, 
\[
  \cyc^{-1}(Z)=\{F\in\cF(\cM)\mid \cyc(F)=Z\}\ \text{ and }\
  \cl^{-1}(Z)=\{V\in\cL(E)\mid \cl(V)=Z\},
\]
Furthermore,  the \textbf{cloud polynomial} and \textbf{flock polynomial} of~$Z$ are,  respectively,
\[
  c_{\cM,Z}=\sum_{F\in\cyc^{-1}(Z)}x^{\hat{\rho}-\rho(F)}\in\Z[x]\ \text{ and }\ 
  f_{\cM,Z}=\sum_{V\in\cl^{-1}(Z)}y^{\dim V-\rho(V)}\in\Z[y].
\]
\end{defi}

Note the asymmetry in the above definition.
Whereas the cloud consists of the \emph{flats} whose cyclic core is~$Z$, the flock consists of \emph{all} spaces whose closure is~$Z$. 
We will discuss this asymmetry later in \cref{R-CloudFlockDual}.

\begin{rem}\label{R-DegCloudFlock}
It is easy to see that $\deg c_{\cM,Z}=\hat{\rho}-\rho(Z)$, which is the corank of~$Z$, and $\deg f_{\cM,Z}=\dim Z-\rho(Z)$, which is the nullity of~$Z$ (for the monotonicity of the nullity function see \cite[p.~99]{AlBy24}).
Moreover, if~$\cM$ is full (see \cref{D-FlatsOpenCyclic}), then $\cyc^{-1}(E)=\{E\}$ and $\cl^{-1}(0)=\{0\}$ and thus $c_{\cM,E}=1=f_{\cM,0}$.
\end{rem}

The following simple descriptions of the cloud and flock will be very useful.

\begin{prop}\label{P-CloudFlockProp}
Let $Z\in\cZ(\cM)$. Then 
\begin{align*}
    \cl^{-1}(Z)&=\{V\in\cL(E)\mid V\leq Z\text{ and }\rho(V)=\rho(Z)\},\\
    \cyc^{-1}(Z)&=\{F\in\cF(\cM)\mid Z\leq F\text{ and } \dim F-\rho(F)=\dim Z-\rho(Z)\}.
\end{align*}
\end{prop}

\begin{proof}
The statement about $\cl^{-1}(Z)$ is clear from the definition of the closure and the fact that~$Z$ is a flat.
For the cloud we use duality; see  \cref{T-DualqM} and \cref{R-cyccl}.  
Denoting by $\cl^*$ the closure in the dual \qM{} $\cM^*$, 
we obtain for any $F\in\cF(\cM)$
\begin{align*}
   F\in\cyc^{-1}(Z)&\Longleftrightarrow Z=\cyc(F)\Longleftrightarrow Z^\perp=\cyc(F)^\perp\Longleftrightarrow  Z^\perp=\cl^*(F^\perp)\\
       &\Longleftrightarrow F^\perp\leq Z^\perp\text{ and } \rho^*(Z^\perp)=\rho^*(F^\perp)\\
        &\Longleftrightarrow Z\leq F\text{ and } \dim Z^\perp+\rho(Z)-\rho(E)=\dim F^\perp+\rho(F)-\rho(E)\\
      &\Longleftrightarrow Z\leq F\text{ and } \rho(Z)-\dim Z=\rho(F)-\dim F,     
\end{align*}
where the fourth step follows from the first part.
\end{proof}

\begin{exa}\label{E-CloudFlockUniform}
Let $k\in[n]_0$ and $\cU_{k,n}:=\cU_{k,n}(E)$ be the uniform \qM{} of rank~$k$ on an $n$-dimensional ground space~$E$.
The flats are
$\cF(\cU_{k,n})=\{F\leq E\mid \dim F\leq k-1\text{ or }F=E\}$ and the cyclic spaces are
$\cO(\cU_{k,n})=\{V\leq E\mid \dim V\geq k+1\text{ or }V=0\}$, respectively.
Thus for $1\leq k\leq n-1$ we have $\cZ(\cU_{k,n})=\{0,E\}$, while $\cZ(\cU_{0,n})=\{E\}$ and $\cZ(\cU_{n,n})=\{0\}$;
see also \cite[Ex.~4.3]{GLJ24DSCyc}. 
Furthermore,  
\[
    \cl^{-1}(E)= \{V\in\cL(E)\mid \dim V\geq k\}\quad \text{ for }\ 0\leq k\leq n-1
\]
and
\[
    \cyc^{-1}(0)=\left\{\begin{array}{cl}\{F\in\cL(E)\mid \dim F\leq k-1\},&\text{if }1\leq k\leq n-1,\\  \cL(E),&\text{if }k=n.\end{array}\right.
\]
Define for $k\in[n]_0$ 
\[
    c_{k,n}=\delta_{k,n}+\sum_{j=0}^{k-1}\GaussianD{n}{j}_qx^{k-j} \text{ and }\
    f_{k,n}=\sum_{j=k}^n\GaussianD{n}{j}_q y^{j-k},
\]
where $\delta_{k,n}$ is the Kronecker-$\delta$-function.
Then $c_{k,n}=c_{\cU_{k,n},0}$ for $k\in[n]$ and $f_{k,n}=f_{\cU_{k,n},E}$ for $k\in[n-1]_0$.
\end{exa}

The rest of this section is devoted to showing that the Whitney function is fully determined by the collection of cloud and flock polynomials.
To do so, we need the following simple lemma from Linear Algebra, whose proof can be found in the appendix.

\begin{lemma}\label{L-DirCompl}
Let $F,\,Z$ be $\F_q$-vector spaces with $\dim F=f$ and $\dim Z=z$ and $Z\leq F$. 
\begin{alphalist}
\item Set $\cW=\{W\in\cL(F)\mid Z\oplus W=F\}$, that is,~$\cW$ is the collection of all direct complements of~$Z$ in~$F$. Then $|\cW|=q^{(f-z)z}$.
\item Let $V\leq Z$ and $\dim V=v$. Then the set $\{V\oplus W\mid W\in\cW\}$ has cardinality $q^{(f-z)(z-v)}$ and each subspace
      in $\{V\oplus W\mid W\in\cW\}$ is obtained $q^{(f-z)v}$ times.
\end{alphalist}
\end{lemma}

As a first step we establish a relation between the flock of a flat and the flock of its cyclic core.

\begin{theo}\label{T-gammaMap}
Let $F\in\cF(\cM)$ and $Z=\cyc(F)$, thus $Z\in\cZ(\cM)$. 
Let the collection~$\cW$ be as in \cref{L-DirCompl}(a). 
Then the map
\[
   \gamma:\cl^{-1}(Z)\times\cW\longrightarrow \cl^{-1}(F),\quad (V,W)\longmapsto V\oplus W
\]
is well-defined and surjective.
Moreover, for any $T\in\cl^{-1}(F)$ the fiber is of the form
\[
    \gamma^{-1}(T)=\{(T\cap Z,W)\mid W\in\cW_T\},
\]
for some subset $\cW_T\subseteq\cW$ of size $|\cW_T|=q^{(f-z)\dim(T\cap Z)}$.
In particular, if $T\in\cl^{-1}(F)$ then $T\cap Z\in\cl^{-1}(Z)$ and 
$\dim(T\cap Z)=\dim T-\dim F+\dim Z$.
\end{theo}

\begin{proof}
For any $W\in\cW$ we have $Z\oplus W=F$. Consider the restriction $\hat{\cM}=\cM|_F$ and denote its rank function again by~$\rho$.
Then \cite[Prop.~7.5]{GLJ24DSCyc} implies 
\begin{equation}\label{e-hatM}
    \hat{\cM}=\hat{\cM}|_Z\oplus\hat{\cM}|_W
\end{equation}
and $\hat{\cM}|_W$ is the free \qM{} on the ground space~$W$. Thus, $\rho(F)=\rho(Z)+\dim W$.

\underline{Well-definedness of~$\gamma$:} 
We have to show that $\cl(V\oplus W)=F$ for all $V\in\cl^{-1}(Z)$ and $W\in\cW$. 
Using \eqref{e-RhoDirect} and the fact that $\cl(V)=Z$ we obtain from \eqref{e-hatM} 
\[
  \rho(V\oplus W)=\rho(V)+\rho(W)=\rho(V)+\dim W=\rho(Z)+\dim W=\rho(F).
\]
Since $V\oplus W\leq F$ and~$F$ is a flat, this implies $\cl(V\oplus W)=F$.
\\
\underline{Surjectivity:} 
Let $T\in\cl^{-1}(F)$. Set $V=T\cap Z$. 
Furthermore, let $Y\leq T$ be such that $T=V\oplus Y$.
Then there exists $W\in\cW$ such that $Y\leq W$. 
(Indeed, let $\beta_1,\beta_2,\beta_3$ be such that~$\beta_1$ is a basis of $T\cap Z$, and $\beta_2$ is a basis of $Y$, and 
$\beta_1\cup\beta_3$ is a basis of~$Z$. 
Then $\beta_1\cup\beta_2$ is a basis of $T$ and $\beta_1\cup\beta_2\cup\beta_3$ is a basis of $T+Z$.
In particular, the last set is linearly independent. 
Hence $\beta_2$ can be extended to a basis of a direct complement of~$Z$.)
Now we obtain from \eqref{e-hatM} along with \eqref{e-RhoDirect}
\[
  \rho(V)+\dim Y=\rho(T)=\rho(F)=\rho(Z)+\dim W.
\]
Since $\dim Y\leq\dim W$ and $\rho(V)\leq \rho(Z)$, this implies $\dim Y=\dim W$ and thus $Y=W$. 
Moreover, $\rho(V)=\rho(Z)$ and thus $V\in\cl^{-1}(Z)$. 
Now we have $T=\gamma(V,W)$ and $\dim T=\dim(T\cap Z)+\dim F-\dim Z$.
This also shows that the fiber $\gamma^{-1}(T)$ is of the stated form, and its size is a consequence of \cref{L-DirCompl}(b).
\end{proof}

The following is a consequence of \cref{T-gammaMap} and \cref{L-DirCompl}(b).

\begin{cor}\label{C-SizeFlock}
Let $F\in\cF(\cM)$ and $Z=\cyc(F)$, thus $Z\in\cZ(\cM)$.  Moreover, let $\dim F=f$ and $\dim Z=z$.
Then 
\[
  \big|\{T\in\cl^{-1}(F)\mid \dim T=v+f-z\}\big|=q^{(f-z)(z-v)}\big|\{V\in\cl^{-1}(Z)\mid \dim V=v\}\big|.
\]
\end{cor}

Now we are ready to present a $q$-analogue of the cloud-flock formula for the Whitney function of matroids;
see \cite[Thm.~2.1]{Ebe14} or \cite[Prop.~4.6]{PlBae14}.
Note that -- different from the matroid case -- the Whitney function is not simply the sum of the products of 
the cloud and flock polynomials over all cyclic flats of~$\cM$.

\begin{theo}\label{T-CloudFlock}
\[
   R_\cM=\sum_{Z\in\cZ(\cM)}\sum_{F\in\cyc^{-1}(Z)}x^{\hat{\rho}-\rho(F)}\sum_{V\in\cl^{-1}(Z)}q^{(\dim F-\dim Z)(\dim Z-\dim V)}y^{\dim V-\rho(Z)}.
\]
\end{theo}

\begin{proof}
Let us first fix some $F\in\cF(\cM)$ and let $Z=\cyc(F)$. Let $\dim F=f$ and $\dim Z=z$. 
Then $f-\rho(F)=z-\rho(Z)$ thanks to \cref{R-cyccl}(a).
Using \cref{T-gammaMap} and \cref{C-SizeFlock} we first compute
\begin{align*}
  \sum_{T\in\cl^{-1}(F)}y^{\dim T-\rho(T)}&=\sum_{v=0}^z\sum_{T\in\cl^{-1}(F)\atop \dim(T\cap Z)=v}y^{v+f-z-\rho(F)}
      =\sum_{v=0}^zq^{(f-z)(z-v)}\sum_{V\in\cl^{-1}(Z)\atop\dim V=v}y^{v+f-z-\rho(F)}\\
    &=\sum_{v=0}^zq^{(f-z)(z-v)}\sum_{V\in\cl^{-1}(Z)\atop\dim V=v}y^{v-\rho(Z)}\\
    &=\sum_{V\in\cl^{-1}(Z)}q^{(f-z)(z-\dim V)}y^{\dim V-\rho(Z)}.
\end{align*}
From this we obtain
\begin{align*}
   R_\cM&=\sum_{T\in\cL(E)}x^{\hat{\rho}-\rho(T)}y^{\dim T-\rho(T)}=
       \sum_{F\in\cF(\cM)}\sum_{T\in\cl^{-1}(F)}x^{\hat{\rho}-\rho(T)}y^{\dim T-\rho(T)}\\
       &=\sum_{F\in\cF(\cM)}x^{\hat{\rho}-\rho(F)}\sum_{T\in\cl^{-1}(F)}y^{\dim T-\rho(T)}\\
       &=\sum_{F\in\cF(\cM)}x^{\hat{\rho}-\rho(F)}\sum_{V\in\cl^{-1}(\cyc(F))}q^{(\dim F-\dim\cyc(F))(\dim\cyc(F)-\dim V)}y^{\dim V-\rho(\cyc(F))}\\
       &=\sum_{Z\in\cZ(\cM)}\sum_{F\in\cyc^{-1}(Z)}x^{\hat{\rho}-\rho(F)}\sum_{V\in\cl^{-1}(Z)}q^{(\dim F-\dim Z)(\dim Z-\dim V)}y^{\dim V-\rho(Z)}.
       \qedhere
\end{align*}
\end{proof}

It is not immediate from the above result that the Whitney function depends only on the cloud and flock polynomials of the cyclic flats of~$\cM$.
Fortunately, this is indeed the case as we can see from the following expression.

\begin{cor}\label{C-CloudFlock2}
For any $Z\in\cZ(\cM)$ write 
\[
   c_{\cM,Z}=\sum_{\alpha=0}^{\deg c_{\cM,Z}}c_{\alpha,Z}x^\alpha\ \text{ and }\ 
   f_{\cM,Z}=\sum_{\beta=0}^{\deg f_{\cM,Z}}f_{\beta,Z}y^\beta.
\]
Then
\[
  R_{\cM}=\sum_{Z\in\cZ(\cM)}\sum_{\alpha=0}^{\deg c_{\cM,Z}}\sum_{\beta=0}^{\deg f_{\cM,Z}}
  q^{(\deg c_{\cM,Z}-\alpha)(\deg f_{\cM,Z}-\beta)}
                c_{\alpha,Z}f_{\beta,Z}x^\alpha y^\beta.
\]
\end{cor}

\begin{proof}
Consider $R_\cM$ as in \cref{T-CloudFlock}.
Let $Z\in\cZ(\cM)$.
Thanks to Remarks~\ref{R-DegCloudFlock} and~\ref{R-cyccl}(a) we have for any $F\in\cyc^{-1}(Z)$ and $V\in\cl^{-1}(Z)$
\begin{align*}
    &\dim F-\dim Z=\rho(F)-\rho(Z)=\hat{\rho}-\rho(Z)-(\hat{\rho}-\rho(F))=\deg c_{\cM,Z}-(\hat{\rho}-\rho(F)),\\
    &\dim Z-\dim V=\dim Z-\rho(Z)-(\dim V-\rho(V))=\deg f_{\cM,Z}-(\dim V-\rho(Z)).
\end{align*}
Now the result follows from \cref{T-CloudFlock} and \cref{D-CloudFlock}.
\end{proof}

The last result can be expressed more elegantly by defining a suitable product between polynomials in $\Z[x]$ and $\Z[y]$.

\begin{rem}\label{R-CloudFlock3}
For $f=\sum_{i=0}^N f_i x^i\in\Z[x]$ and $g=\sum_{j=0}^M g_jy^j\in\Z[y]$ define
\[
     f*g:=\sum_{i=0}^N\sum_{j=0}^M q^{(\deg f-i)(\deg g-j)}f_ig_jx^iy^j.
\]
Then \cref{C-CloudFlock2} reads as $R_\cM=\sum_{Z\in\cZ(\cM)} c_{\cM,Z}* f_{\cM,Z}$.
\end{rem}

We close this section with the following example.

\begin{exa}\label{E-CloudFlock2}
Consider the field $\F_{2^7}$ with primitive element~$\omega$ satisfying $\omega^7+\omega+1=0$. 
Let 
\[
   G=\begin{pmatrix}1&\omega^{90}&0&\omega^{10}&0&\omega^4\\0&0&1&\omega^7&0&\omega^{90}\\0&0&0&0&1&\omega^{32}\end{pmatrix}
\]
and $\cM=\cM_G=(\F_2^6,\rho)$.
Using \cref{D-Polys} one obtains
\begin{equation}\label{e-WhitExa1}
    R_\cM=x^3 + (2y + 63)x^2 + (y^2 + 42y + 649)x + y^3 + 63y^2 + 650y + 1353.
\end{equation}
Furthermore, there are 8 cyclic flats and they are given by
\begin{align*}
    &Z_0=0,\ Z_1=\subspace{e_1,e_2},\ Z_2=\subspace{e_1+e_2+e_4,e_3+e_4},\\
    &Z_3=\subspace{e_1+e_3+e_4,e_2+e_4+e_5,e_6},\ Z_4=\subspace{e_1+e_4+e_5,e_2+e_3+e_5,e_6},\\
    &Z_5=\subspace{e_2+e_5,e_3+e_5+e_6,e_4},\ Z_6=\subspace{e_1,e_2,e_3,e_4},\ Z_7=\F_2^6.
\end{align*}
Hence the lattice of cyclic flats has the form
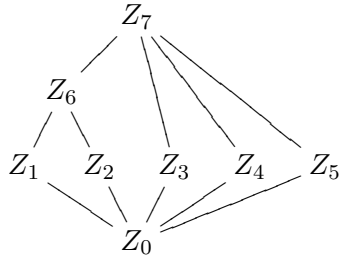
\[
\begin{array}{l}
   \begin{xy}
   (25,0)*+{Z_0}="c0";%
   (10,10)*+{Z_1}="c1";%
   (20,10)*+{Z_2}="c2";%
   (30,10)*+{Z_3}="c3";%
   (40,10)*+{Z_4}="c4";%
   (50,10)*+{Z_5}="c5";%
   (15,20)*+{Z_6}="c6";
   (25,30)*+{Z_7}="c7";%
    {\ar@{-}"c0";"c1"};
    {\ar@{-}"c0";"c2"};
    {\ar@{-}"c0";"c3"};
    {\ar@{-}"c0";"c4"};
    {\ar@{-}"c0";"c5"};
    {\ar@{-}"c1";"c6"};
    {\ar@{-}"c2";"c6"};
    {\ar@{-}"c3";"c7"};
    {\ar@{-}"c4";"c7"};
    {\ar@{-}"c5";"c7"};    
    {\ar@{-}"c6";"c7"};
   \end{xy}
\end{array}
\]
(for the lattice structure of $\cZ(\cM)$ see \cite[Prop.~3.24]{AlBy24} or \cite[Cor.~4.2]{GLJ24DSCyc}). According to \cref{C-CloudFlock2} we have $R_{\cM}=\sum_{Z\in\cZ(\cM)} h_{\cM,Z}$, where 
\begin{equation}\label{e-hmz}
   h_{\cM,Z}=\sum_{\alpha=0}^{\deg c_{\cM,Z}}\sum_{\beta=0}^{\deg f_{\cM,Z}}
                  q^{(\deg c_{\cM,Z}-\alpha)(\deg f_{\cM,Z}-\beta)}c_{\alpha,Z}f_{\beta,Z}x^\alpha y^\beta
\end{equation}
with the notation as in that corollary.
We obtain the following data
\[
        \begin{array}{|c||c|c|c|}
        \hline
         \text{cyclic flat }&c_{\cM,Z_i}&f_{\cM,Z_i}&h_{\cM,Z_i}\\ \hline\hline
         Z_0 & x^3 + 57x^2 + 451x   & 1 & x^3 + 57x^2 + 451x\\ \hline
         Z_1&x^2 + 12 x & y+3 &(y + 3)x^2 + (12y + 72)x \\ \hline
         Z_2& x^2+12x &y+3 &(y + 3)x^2 + (12y + 72)x\\ \hline
         Z_3&x& y+7& (y + 7)x \\ \hline 
         Z_4&x & y+7&(y + 7)x \\ \hline 
         Z_5&x &y+7 & (y + 7)x \\ \hline 
         Z_6&x & y^2 + 15y + 33 &(y^2 + 15y + 33)x \\ \hline 
         Z_7&1& y^3 + 63y^2 + 650y + 1353 & y^3 + 63y^2 + 650y + 1353\\ \hline
         \end{array}
\]
The sum of the last column is indeed in $R_\cM$ as given in \eqref{e-WhitExa1}.
\end{exa}

\section{Configurations and Cloud-Flock Lattices}\label{S-Config}

Recall that $\cZ(\cM)$ is a lattice whose vertices are labeled by the cyclic flats.
Together with the rank values of the cyclic flats it fully determines the entire \qM{}; see \eqref{e-RhoqM}.
We will show below that for full \qM{}s the Whitney function, and in fact the cloud and flock polynomials, are determined by 
less information, namely the configuration of~$\cM$, defined as follows.
The result will be generalized to arbitrary \qM{}s in \cref{S-DirSums}.
For matroids, configurations were introduced in \cite[Def.~2]{Ebe14}.

\begin{defi}\label{D-Config}
Let $\cM$ be a \qM{} of rank~$\hat{\rho}$. 
\begin{alphalist}
\item For $Z\in\cZ(\cM)$ set $\lambda(Z)=(\hat{\rho}-\rho(Z),\dim Z-\rho(Z))$, that is $\lambda(Z)$ is 
      the \textbf{corank-nullity pair} of~$Z$.
      Define the \textbf{configuration} of~$\cM$, denoted by $\Config(\cM)$, as the lattice with underlying set $\{\lambda(Z)\mid Z\in\cZ(\cM)\}$ and the partial order $\lambda(Z)\leq \lambda(Z')$ if $Z\leq Z'$. Thus 
      \[
        \lambda(Z)\wedge\lambda(Z')=\lambda(Z\wedge Z')\ \text{ and }\
        \lambda(Z)\vee\lambda(Z')=\lambda(Z\vee Z').
      \]
\item Define the \textbf{cloud-flock lattice} of~$\cM$, denoted by $\CF(\cM)$, as the lattice with underlying set 
      $\{(c_{\cM,Z},f_{\cM,Z})\mid Z\in\cZ(\cM)\}$ and the partial $(c_{\cM,Z},f_{\cM,Z})\leq (c_{\cM,Z'},f_{\cM,Z'})$ if $Z\leq Z'$. Thus
      \begin{align*}
        (c_{\cM,Z},f_{\cM,Z})\wedge (c_{\cM,Z'},f_{\cM,Z'})&=(c_{\cM,Z\wedge Z'},f_{\cM,Z\wedge Z'}),\\
        (c_{\cM,Z},f_{\cM,Z})\vee (c_{\cM,Z'},f_{\cM,Z'})&=(c_{\cM,Z\vee Z'},f_{\cM,Z\vee Z'}).
      \end{align*}
\end{alphalist}
\end{defi}

By definition, $\cZ(\cM),\,\CF(\cM)$, and $\Config(\cM)$ are isomorphic lattices. 

\begin{exa}\label{E-ConfigCF}
Consider \cref{E-CloudFlock2}. Then $\CF(\cM)$ and $\Config(\cM)$ are given by 
\[
\begin{array}{l}
   \begin{xy}
   (25,0)*+{\mbox{${\scriptstyle(x^3\!+\!57x^2\!+\!451x,\,1)}$}}="c0";%
   (0,10)*+{\mbox{${\scriptstyle(x^2\!+\!12x,\,y\!+\!3)}$}}="c1";%
   (20,10)*+{\mbox{${\scriptstyle(x^2\!+\!12x,\,y\!+\!3)}$}}="c2";%
   (40,10)*+{\mbox{${\scriptstyle(x,\,y\!+\!7)}$}}="c3";%
   (60,10)*+{\mbox{${\scriptstyle(x,\,y\!+\!7)}$}}="c4";%
   (80,10)*+{\mbox{${\scriptstyle(x,\,y\!+\!7)}$}}="c5";%
   (10,20)*+{\mbox{${\scriptstyle(x,\,y^2\!+\!15y\!+\!33)}$}}="c6";%
   (25,30)*+{\mbox{${\scriptstyle(1,\,y^3\!+\!63y^2\!+\!650y\!+\!1353)}$}}="c7";%
    {\ar@{-}"c0";"c1"};
    {\ar@{-}"c0";"c2"};
    {\ar@{-}"c0";"c3"};
    {\ar@{-}"c0";"c4"};
    {\ar@{-}"c0";"c5"};
    {\ar@{-}"c1";"c6"};
    {\ar@{-}"c2";"c6"};
    {\ar@{-}"c3";"c7"};
    {\ar@{-}"c4";"c7"};
    {\ar@{-}"c5";"c7"};    
    {\ar@{-}"c6";"c7"};
   \end{xy}
\end{array}\hspace*{6cm}\mbox{}
\]

\vspace*{-1.6cm}
\[
\mbox{}\hspace*{6cm}
\begin{array}{l}
   \begin{xy}
   (25,0)*+{\mbox{$(3,0)$}}="c0";%
   (0,10)*+{\mbox{$(2,1)$}}="c1";%
   (20,10)*+{\mbox{$(2,1)$}}="c2";%
   (40,10)*+{\mbox{$(1,1)$}}="c3";%
   (60,10)*+{\mbox{$(1,1)$}}="c4";%
   (80,10)*+{\mbox{$(1,1)$}}="c5";%
   (10,20)*+{\mbox{$(1,2)$}}="c6";%
   (25,30)*+{\mbox{$(0,3)$}}="c7";%
    {\ar@{-}"c0";"c1"};
    {\ar@{-}"c0";"c2"};
    {\ar@{-}"c0";"c3"};
    {\ar@{-}"c0";"c4"};
    {\ar@{-}"c0";"c5"};
    {\ar@{-}"c1";"c6"};
    {\ar@{-}"c2";"c6"};
    {\ar@{-}"c3";"c7"};
    {\ar@{-}"c4";"c7"};
    {\ar@{-}"c5";"c7"};    
    {\ar@{-}"c6";"c7"};
   \end{xy}
\end{array}
\]
\end{exa}

Clearly, $\cZ(\cM)$ determines $\CF(\cM)$, which in turn determines $\Config(\cM)$ because the corank and nullity are 
the degrees of the cloud and flock polynomials, respectively; see \cref{R-DegCloudFlock}.
Note that $\CF(\cM)$ carries strictly less information than the lattice $\cZ(\cM)$.
For instance, the non-equivalent \qM{}s in \cref{E-NonIsoWhitney} turn out to have the same CF-lattices.
We will see another instance in \cref{E-CloudFlock3}.

In the main result of this section we will prove that the configuration determines the cloud-flock lattice, and thus the Whitney function.
In order to do so, we need some preparation.
Recall the Correspondence Theorem from Linear Algebra:
Let $W$ be a subspace of the vector space~$V$.
Then $X\mapsto X/W$ provides us with a lattice isomorphism between 
$\cL_W(V):=\{X\in\cL(V)\mid W\leq X\}$ and $\cL(V/W)$, where in either case the meet (resp., join) is the intersection (resp., sum).
This lattice isomorphism immediately leads to descriptions of the flats and cyclic spaces of certain restrictions and contractions.
The straightforward proof of the following result can be found in the appendix.
For restrictions and contractions recall \cref{D-RestrContr}.
 
\begin{lemma}\label{L-RestrContr}
Let $\hat{F}\in\cF(\cM),\,\hat{O}\in\cO(\cM)$, and $\hat{Z}\in\cZ(\cM)$.
\begin{alphalist}
\item $\cF(\cM|\hat{F})=\{F\in\cF(\cM)\mid F\leq \hat{F}\}$;
\item $\cF(\cM/\hat{F})=\{F/\hat{F}\mid F\in\cF(\cM),\,\hat{F}\le F\}$;
\item $\cO(\cM|\hat{O})=\{O\mid O\in\cO(\cM),\,O\leq\hat{O}\}$;
\item $\cO(\cM/\hat{O})=\{O/\hat{O}\mid O\in\cO(\cM),\,\hat{O}\leq O\}$.
\item Set $\cM'=\cM/\hat{Z}$. Then $0\in\cZ(\cM')$ and $\cyc_{\cM'}^{-1}(0)=\{F/\hat{Z}\mid F\in\cyc_{\cM}^{-1}(\hat{Z})\}$.
\end{alphalist}
\end{lemma}

\begin{cor}\label{C-CloudFlockRestrContr}
Let $\hat{Z}\in\cZ(\cM)$. Then $c_{\cM,\hat{Z}}=c_{\cM/\hat{Z},0}$ and $f_{\cM,\hat{Z}}=f_{\cM|\hat{Z},\hat{Z}}$.
\end{cor}

\begin{proof}
The identity for the flock polynomial is obvious.
For the cloud polynomial let $\cM'=\cM/\hat{Z}=(E/\hat{Z},\rho')$.
Then $c_{\cM',0}=\sum_{F/\hat{Z}\in\cyc^{-1}(0)}x^{\rho'(E/\hat{Z})-\rho'(F/\hat{Z})}=\sum_{F/\hat{Z}\in\cyc^{-1}(0)}x^{\hat{\rho}-\rho(F)}$
and the result follows from \cref{L-RestrContr}(e).
\end{proof}

The next result will be needed for an inductive argument in the proof of the main result of this section, \cref{T-ConfigWhitney}.
For any $W_1,\,W_2\in\cL(E)$ with $W_1\leq W_2$ define the interval $[W_1,\,W_2]:=\{V\in\cL(E)\mid W_1\leq V\leq W_2\}$. 

\begin{prop}\label{P-LatticeIso}
Let $Z_1,\,Z_2\in\cZ(\cM)$ and $Z_1\leq Z_2$. 
\begin{alphalist}
\item $[Z_1,\,Z_2]\cap\cZ(\cM)$ is a lattice with $Z'\wedge Z''=\cyc(Z'\cap Z'')$ and $Z'\vee Z''=\cl(Z'+ Z'')$.
      Thus it is a sublattice of $\cZ(\cM)$. 
\item Let $\widehat{\cM}=(\cM|Z_2)/Z_1$. Then the map
      \[
         \phi:[Z_1,\,Z_2]\cap\cZ(\cM)\longrightarrow\cZ(\widehat{\cM}),\quad Z\longmapsto Z/Z_1
      \]
      is a lattice isomorphism. 
\end{alphalist}
\end{prop}

\begin{proof}
(a) Let $Z',Z''\in[Z_1,\,Z_2]\cap\cZ(\cM)$. 
Since $Z_1$ is cyclic and $Z_1\leq Z'\cap Z''$, it follows that $Z_1\leq\cyc(Z'\cap Z'')$; see \cite[Thm.~3.6]{GLJ24DSCyc}.
Similarly, $Z_2$ is a flat containing $Z'+Z''$ and thus~$Z_2$ contains $\cl(Z'+Z'')$.
Hence $Z'\wedge Z''$ and $Z'\vee Z''$ are in $[Z_1,\,Z_2]\cap\cZ(\cM)$.
\\
(b) With the aid of \cref{L-RestrContr} we obtain 
$\cZ(\widehat{\cM})=\{Z/Z_1\mid Z\in\cZ(\cM|Z_2),Z_1\leq Z\}=\{Z/Z_1\mid Z\in\cZ(\cM),\,Z_1\leq Z\leq Z_2\}$, and this shows that $\phi$ is a
well-defined bijection.
It remains to show that~$\phi$ is meet- and join-preserving. 
Let $Z',Z''\in[Z_1,\,Z_2]\cap\cZ(\cM)$. With the aid of the correspondence theorem and \cref{L-RestrContr} we have for any $F\in\cF(\cM)$
\begin{align*}
    F=Z'\vee Z''&\Longleftrightarrow F=\cl(Z'+Z'')\\
       &\Longleftrightarrow F\text{ is the inclusion-smallest flat of~$\cM|Z_2$ containing }Z'+Z''\\
       &\Longleftrightarrow F/Z_1\text{ is the inclusion-smallest flat of $\widehat{\cM}$ containing }(Z'+Z'')/Z_1\\
      &\Longleftrightarrow\phi(F)\text{ is the inclusion-smallest flat of $\widehat{\cM}$ containing }\phi(Z')+\phi(Z'')\\
     &\Longleftrightarrow\phi(F)=\cl(\phi(Z')+\phi(Z''))=\phi(Z')\vee\phi(Z'').
\end{align*}
In the same way, using that $Z'\wedge Z''=\cyc(Z'\cap Z'')$ is the inclusion-largest cyclic space contained in $Z'\cap Z''$, we obtain that 
$\phi(\cyc(Z'\cap Z''))=\phi(Z')\wedge\phi(Z'')$.
\end{proof}

In the rest of this section we restrict ourselves to full \qM{}s~$\cM$; see \cref{D-FlatsOpenCyclic}.
We start with showing that the cloud and flock polynomial $c_{\cM,0}$ and $f_{\cM,E}$ are determined by the 
remaining cloud and flock polynomials.
To do so, we need the following maps.

\begin{defi}\label{D-truncation}
Define 
\begin{align*}
    &d_x:\Z[x,y]\longrightarrow\Z[x],\quad \sum_{i=0}^n\sum_{j=0}^mf_{ij}x^iy^j\longmapsto \sum_{i=0}^n\sum_{j=0}^{i-1}f_{ij}x^{i-j},\\[.5ex]
    &d_y:\Z[x,y]\longrightarrow\Z[y],\quad \sum_{i=0}^n\sum_{j=0}^mf_{ij}x^iy^j\longmapsto \sum_{i=0}^n\sum_{j=i}^{m}f_{ij}y^{j-i}.
\end{align*}
Note that $d_x(f)$ is the truncation of the Laurent polynomial $f(x,x^{-1})$ to the terms with positive exponents, while $d_y(f)$ is the truncation of 
$f(y^{-1},y)$ to the terms with non-negative exponents.
Both maps are $\Z$-linear.
\end{defi}

Recall \cref{R-CloudFlock3}.

\begin{prop}\label{P-TruncWhitney}
Let $\cM=(E,\rho)$ be a full \qM{} of rank~$\hat{\rho}$ on the $n$-dimensional ground space~$E$.
Set $\widehat{\cZ}(\cM)=\cZ(\cM)\setminus\{0,E\}$ and 
$\widehat{R}_{\cM}=\sum_{Z\in\widehat{\cZ}(\cM)}c_{\cM,Z}*f_{\cM,Z}$.
Then 
\[
     c_{\hat{\rho},n}=d_x(\widehat{R}_{\cM})+c_{\cM,0}\quad \text{ and }\quad
      f_{\hat{\rho},n}=d_y(\widehat{R}_{\cM})+f_{\cM,E},
\]
where $c_{\hat{\rho},n}$ and $f_{\hat{\rho},n}$ are the cloud and flock polynomial of the uniform \qM{} given in \cref{E-CloudFlockUniform}.
Hence $c_{\cM,0}$ and $f_{\cM,E}$ are fully determined by the collection 
$\{(c_{\cM,Z},f_{\cM,Z})\mid Z\in\widehat{\cZ}(\cM)\}$.
\end{prop}

\begin{proof}
First of all, \cref{D-Polys} implies $d_x(R_{\cM})=\sum_{j=0}^{\hat{\rho}-1}\Gaussian{n}{j}_qx^{\hat{\rho}-j}=c_{\hat{\rho},n}$.
Next, since~$\cM$ is full we have $\cl(0)=0$ and $\cyc(E)=E$.
This implies that every $F\in\cyc^{-1}(0)$ satisfies $\rho(F)<\hat{\rho}$. 
Indeed, suppose $\rho(F)=\hat{\rho}=\rho(E)$. 
Then, since~$F$ is a flat, this implies $F=\cl(F)=E$. But then $0=\cyc(F)$ contradicts $\cyc(E)=E$.
Hence $\rho(F)<\hat{\rho}$.
Moreover, every $V\in\cl^{-1}(E)$ satisfies $\dim V\geq\rho(V)=\hat{\rho}$.
All of this together with the $\Z$-linearity of $d_x$ leads to
\begin{align*}
    c_{\hat{\rho},n}&=d_x(R_{\cM})=d_x(\widehat{R}_{\cM})+d_x(\sum_{F\in\cyc^{-1}(0)}x^{\hat{\rho}-\rho(F)}y^0)
         +d_x(\sum_{V\in\cl^{-1}(E)}x^{\hat{\rho}-\hat{\rho}}y^{\dim V-\hat{\rho}})\\
         &=d_x(\widehat{R}_{\cM})+\sum_{F\in\cyc^{-1}(0)}x^{\hat{\rho}-\rho(F)}+d_x(\sum_{V\in\cl^{-1}(E)}x^{\hat{\rho}-\dim V})\\       
         &=d_x(\widehat{R}_{\cM})+c_{\cM,0}+0,
\end{align*}
where the third and fourth step follows from the facts that every exponent in the terms $x^{\hat{\rho}-\rho(F)}$ is positive while every exponent in the terms
$x^{\hat{\rho}-\dim V}$ is non-positive.
\\
Similarly, $d_y(R_{\cM})=\sum_{j=\hat{\rho}}^{n}\Gaussian{n}{j}_qy^{j-\hat{\rho}}=f_{\hat{\rho},n}$ and we obtain
\begin{align*}
    f_{\hat{\rho},n}&=d_y(\widehat{R}_{\cM})+d_y(\sum_{F\in\cyc^{-1}(0)}x^{\hat{\rho}-\rho(F)}y^0)
         +d_y(\sum_{V\in\cl^{-1}(E)}x^{\hat{\rho}-\hat{\rho}}y^{\dim V-\hat{\rho}})\\
       &=d_y(\widehat{R}_{\cM})+d_y(\sum_{F\in\cyc^{-1}(0)}y^{\rho(F)-\hat{\rho}}) +d_y(\sum_{V\in\cl^{-1}(E)}y^{\dim V-\hat{\rho}})\\
       &=d_y(\widehat{R}_{\cM})+0+\sum_{V\in\cl^{-1}(E)}y^{\dim V-\hat{\rho}}=d_y(\widehat{R}_{\cM})+f_{\cM,E},
\end{align*}
where the third step follows from the fact that every exponent in $y^{\rho(F)-\hat{\rho}}$ is negative and every exponent in $y^{\dim V-\hat{\rho}}$ is non-negative.
This proves the two identities. 
The last statement is clear.
\end{proof}

Now we are ready for our main result, which is the $q$-analogue of \cite[Thm.~4.1]{Ebe14}.

\begin{theo}\label{T-ConfigWhitney}
Let $\cM=(E,\rho)$ be a full \qM{}.
Then $\Config(\cM)$ fully determines $\CF(\cM)$.
As a consequence, the configuration determines the Whitney function.
\end{theo}

In the \cref{S-DirSums} we will generalize \cref{T-ConfigWhitney} to general \qM{}s, and in \cref{S-Converse} we will discuss the converse direction, that is, how much information about the configuration can be retrieved from the Whitney function.

\begin{proof}
The consequence about the Whitney function follows from \cref{C-CloudFlock2} once we have determined $\CF(\cM)$.
\\
For the first statement we induct on the cardinality of $\cZ(\cM)$. 
Since the abstract lattice of $\CF(\cM)$ is isomorphic to $\Config(\cM)$, we only need to determine the pairs 
$(c_{\cM,Z},f_{\cM,Z})$ for all $Z\in\cZ(\cM)$.
Note that $|\cZ(\cM)|\geq 2$ since $\cM$ is full and furthermore $c_{\cM,E}=1=f_{\cM,0}$.
Furthermore, $\cM$ is not free or trivial, that is, $\hat{\rho}:=\rho(\cM)\not\in\{0,n\}$, where $n=\dim E$; see \cref{E-CloudFlockUniform} (those cases would be trivial anyways).
Since~$E$ and~$0$ are the greatest and least element of $\cZ(\cM)$, the top and bottom vertex of $\Config(\cM)$ are 
labeled by $(0,n-\hat{\rho})$ and $(\hat{\rho},0)$, respectively, and thus~$\hat{\rho}$ and~$n$ are determined by $\Config(\cM)$.
\\
1) Let $|\cZ(\cM)|= 2$. Then $\cZ(\cM)=\{0,E\}$ and $\cM=\cU_{\hat{\rho},n}(E)$ (see \cite[Ex.~7.3(c)]{GLJ24DSCyc}).
In this case $\Config(\cM)$ consists of the top and bottom vertex only, and $c_{\cM,0}$  and $f_{\cM,E}$ are given in \cref{E-CloudFlockUniform}.
\\
2)  Let $|\cZ(\cM)|>2$. 
\\
First choose a vertex in $\Config(\cM)$ different from the top and bottom one. 
It belongs to a cyclic flat~$Z'\in\cZ(\cM)\setminus\{0,E\}$.
While we do not know~$Z'$ itself, we know $(r',n'):=(\rho(Z'),\dim Z')$ from  $\Config(\cM)$.
We want to determine $c_{\cM,Z'}$ and $f_{\cM,Z'}$. 
From \cref{C-CloudFlockRestrContr} we know that 
\begin{equation}\label{e-CloudFlockSmaller}
   c_{\cM,Z'}=c_{\cM/Z',0}\ \text{ and }\ f_{\cM,Z'}=f_{\cM|Z',Z'}.
\end{equation}
Thanks to \cref{P-LatticeIso}, the lattices $\cZ(\cM/Z')$ and $\cZ(\cM|Z')$ are
isomorphic to $[Z',E]\cap\cZ(\cM)$ and $[0,Z']\cap\cZ(\cM)$, respectively.
The latter are sublattices of $\cZ(\cM)$ and hence their isomorphism classes are fully determined by $\Config(\cM)$.
Since in the sublattice $[Z',E]\cap\cZ(\cM)$, the ranks and dimensions go down by~$r'$ and $n'$ (compared to $\cZ(\cM)$), respectively, 
replacing each vertex label $(\hat{\rho}-\rho(Z),\dim Z-\rho(Z))$ in $\Config(\cM)$ by 
$(\hat{\rho}-\rho(Z),\dim Z-\rho(Z)-n'+r')$ gives us  $\Config(\cM/Z')$.
On the other hand, in the sublattice $[0,Z']\cap\cZ(\cM)$ the ranks and dimensions are the same as in $\cZ(\cM)$, and therefore,
replacing each vertex label $(\hat{\rho}-\rho(Z),\dim Z-\rho(Z))$ by $(r'-\rho(Z),\dim Z-\rho(Z))$
gives us $\Config(\cM|Z')$.
Hence $\Config(\cM/Z')$ and $\Config(\cM|Z')$ are fully determined by $\Config(\cM)$.
Since $Z'\not\in\{0,E\}$, the cardinalities of $\cZ(\cM/Z')$ and  $\cZ(\cM|Z')$ are strictly less than $|\cZ(\cM)|$.
By \cref{L-RestrContr} the \qM{}s $\cM/Z'$ and $\cM|Z'$ are full and, using \eqref{e-CloudFlockSmaller}  
we can compute $c_{\cM,Z'}$ and $f_{\cM,Z'}$ from $\Config(\cM)$ by induction.
\\
It remains to determine $c_{\cM,0}$  and $f_{\cM,E}$. 
By the previous part we can compute $c_{\cM,Z'}$ and $f_{\cM,Z'}$ from $\Config(\cM)$ for all $Z'\in\cZ(\cM)\setminus\{0,E\}$.
But then  $c_{\cM,0}$  and $f_{\cM,E}$ are determined thanks to \cref{P-TruncWhitney}.
\end{proof}

\begin{exa}\label{E-CloudFlock3}
Consider the field $\F_{2^7}$ with primitive element~$\omega$ with minimal polynomial $x^7+x+1$. 
Let  
\[
   G_1=\begin{pmatrix}1&0&0&\omega^{65}&\omega^{85}\\0&1&0&\omega^{37}&\omega^{72}\\
                       0&0&1&\omega^{124}&\omega^{118}\end{pmatrix},\
   G_2=\begin{pmatrix}1&0&0&\omega^{26}&\omega^{64}\\0&1&0&\omega^{27}&\omega^{20}\\
                       0&0&1&\omega^{50}&\omega^{92}\end{pmatrix}
\]                       
and set $\cM_i=\cM_{G_i}=(\F_2^5,\rho_i)$. Each~$\cM_i$ has rank $\hat{\rho}=3$ and
$\cZ(\cM_i)=\{0,\F_2^6\}\cup\widehat{\cZ}(\cM_i)$, where
\begin{align*}
   &\widehat{\cZ}(\cM_1)=\big\{\subspace{e_1+e_4,e_2+e_5,e_3+e_4},\subspace{e_1+e_5,e_2+e_5,e_3},\subspace{e_1+e_3+e_4,e_2+e_3,e_5}\big\},\\
   &\widehat{\cZ}(\cM_2)=\big\{\subspace{e_1+e_5,e_2+e_5,e_3+e_4},\subspace{e_1+e_4+e_5,e_2+e_4,e_3+e_4+e_5},\subspace{e_1+e_2+e_4,e_3,e_5}\big\}.
\end{align*}
The \qM{}s $\cM_1$ and $\cM_2$ are not equivalent. 
Indeed, consider the three pairwise intersections of the subspaces in $\widehat{\cZ}(\cM_i)$.
For $i=1$, these intersections result in three distinct 1-dimensional subspaces, whereas for $i=2$ each of these intersections equals
$\subspace{e_1+e_2+e_3+e_4}$.
On the other hand, the lattices $\Config(\cM_i)$ are identical and given by
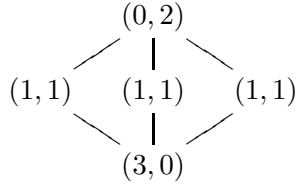
\[
\begin{array}{l}
   \begin{xy}
   (20,0)*+{(3,0)}="c0";%
   (5,10)*+{(1,1)}="c1";%
   (20,10)*+{(1,1)}="c2";%
   (35,10)*+{(1,1)}="c3";%
   (20,20)*+{(0,2)}="c5";%
    {\ar@{-}"c0";"c1"};
    {\ar@{-}"c0";"c2"};
    {\ar@{-}"c0";"c3"};
    {\ar@{-}"c1";"c5"};
    {\ar@{-}"c2";"c5"};
    {\ar@{-}"c3";"c5"};
   \end{xy}
\end{array}
\]
For both \qM{}s we have the following data  (where $h_{\cM,Z}$ is as in \eqref{e-hmz} and $Z_1,Z_2,Z_3$ are the cyclic flats in $\widehat{\cZ}(\cM_i)$).
\[
        \begin{array}{|c||c|c|c|}
        \hline
         \text{cyclic flat }&c_{\cM,Z_i}&f_{\cM,Z_i}&h_{\cM,Z_i}\\ \hline\hline
           0 & x^3 + 31x^2 + 134x & 1 &x^3 + 31x^2 + 134x\\ \hline
         Z_1& x & y+7 &(y + 7)x \\ \hline
         Z_2& x & y+7 &(y + 7)x\\ \hline
         Z_3& x & y+7&(y + 7)x \\ \hline 
         \F_2^5&1 &y^2 + 31y + 152&y^2 + 31y + 152\\ \hline 
          \end{array}
\]
Now \cref{C-CloudFlock2} implies $R_{\cM_i}=x^3 + 31x^2 + (3y + 155)x + y^2 + 31y + 152$ for $i=1,2$.
\end{exa}

The rest of this section is devoted to duality. Recall the dual \qM{} from \cref{T-DualqM}.

\begin{cor}\label{C-DualConfig}
Let $\cM$ be a \qM{} of rank~$\hat{\rho}$ and $\cM^*$ be its dual \qM{}.  
\begin{alphalist}
\item $\Config(\cM^*)$ is obtained from $\Config(\cM)$ by turning the lattice upside down and swapping the entries of the labels. That is, the corank-nullity pair of $Z^\perp$ is $(\dim Z-\rho(Z),\,\hat{\rho}-\rho(Z))$. 
\item If~$\cM$ is full, then $\CF(\cM^*)$ is fully determined by $\CF(\cM)$ (equivalently, by $\Config(\cM)$).
\end{alphalist}
\end{cor}

Later in \cref{C-DualCF} we will establish~(b) for arbitrary \qM{}s.

\begin{proof}
(a) By Remark~\ref{R-FlatsOpen} we have $\cZ(\cM^*)=\{Z^\perp\mid Z\in\cZ(\cM)\}$. Hence the statement follows from the fact that containments of the cyclic flats reverse order together with \eqref{e-coranknullityduality}.
\\
(b) $\CF(\cM)$ determines $\Config(\cM)$, which by~(a) determines $\Config(\cM^*)$. 
Now \cref{T-ConfigWhitney} provides us with $\CF(\cM^*)$.
\end{proof}

Due to the non-explicit proof of \cref{T-ConfigWhitney} it is not obvious how to obtain $\CF(\cM^*)$ directly from~$\CF(\cM)$.

\begin{exa}\label{E-ConfigCFDuality}
Consider again Examples~\ref{E-CloudFlock2} and~\ref{E-ConfigCF}.
The dual \qM{} $\cM^*$ is generated by the matrix
\[
   H=\begin{pmatrix}1&0&0&\omega^{59}&\omega^8&\omega^{103}\\0&1&0&\omega^{22}&\omega^{98}&\omega^{66}\\
         0&0&1&\omega^{100}&\omega^{11}&\omega^{106}\end{pmatrix}.
\]
Its cloud and flock polynomials are given by 
\[
        \begin{array}{|c||c|c|}
        \hline
         \text{cyclic flat }&c_{\cM^*,Z_i^\perp}&f_{\cM^*,Z_i^\perp}\\ \hline\hline
         Z_0^\perp&1 & y^{3} + 63 y^{2} + 649 y + 1353\\ \hline
         Z_1^\perp&x & y^{2} + 15 y + 34\\ \hline 
         Z_2^\perp &x & y^{2} + 15 y + 34 \\ \hline 
         Z_3^\perp&x & y+7 \\ \hline 
         Z_4^\perp&x& y+7 \\ \hline 
         Z_5^\perp& x & y + 7 \\ \hline
         Z_6^\perp&x^{2} + 9 x & y + 3 \\ \hline
         Z_7^\perp & x^{3} + 60 x^{2} + 507 x & 1\\ \hline
         \end{array}
\]
\end{exa}

We wish to briefly discuss the relation between $f_{\cM^*,Z^\perp}$ and $c_{\cM,Z}$.
Recall the definition of the cloud and flock of~$Z$ in \cref{D-CloudFlock}. 
On the one hand, by \cref{R-cyccl}(b) the cloud condition $\cyc(F)=Z$ is equivalent to the dual flock condition $\cl^*(F^\perp)=Z^\perp$.
But on the other hand, the cloud consists only of \emph{flats} whose cyclic core is~$Z$, whereas the dual flock consists of \emph{all} spaces 
whose closure is~$Z^\perp$.
This causes $f_{\cM^*,Z^\perp}$ to differ from $c_{\cM,Z}$.
Even more, there is no map $\psi_{\cM}:\Z[x]\rightarrow\Z[y]$ such that $\psi_{\cM}(c_{\cM,Z})=f_{\cM^*,Z^\perp}$ for all cyclic flats~$Z$. Indeed, 
in \cref{E-ConfigCFDuality} we have $c_{\cM^*,Z_i^\perp}=x$ for $i=1,\ldots,5$, whereas $f_{\cM,Z_i}$ is not independent of $i\in\{1,\ldots,5\}$; see \cref{E-CloudFlock2}.

We can take advantage of the above described asymmetry between the cloud and flock and turn it around by restricting the spaces in the flock instead of in the cloud.
This provides us with further invariants and is described in the following remark.

\begin{rem}\label{R-CloudFlockDual}
Let $\cM=(E,\rho)$ be a \qM{}.
We define the \textbf{supercloud} and \textbf{subflock} of a cyclic flat~$Z$ as
$\CloudU(Z)=\{V\leq E\mid \cyc(V)=Z\}$ and $\FlockD(Z)=\{V\in\cO(\cM)\mid \cl(V)=Z\}$, respectively.
Accordingly, the \textbf{supercloud polynomial} and \textbf{subflock polynomial} are defined as
\[
  c^{\uparrow}_{\cM,Z}=\sum_{V\in\textrm{ cloud}^{\uparrow}(Z)}x^{\rho(E)-\rho(V)}\ \text{ and }\
  f^{\downarrow}_{\cM,Z}=\sum_{V\in\textrm{ flock}^{\downarrow}(Z)}y^{\dim V-\rho(V)}.
\]
Using \cref{R-cyccl}, one easily checks that 
$f_{\cM^*,Z^\perp}=c^{\uparrow}_{\cM,Z}(y)$ and $c_{\cM^*,Z^\perp}=f^{\downarrow}_{\cM,Z}(x)$
for all $Z\in\cZ(\cM)$.
As a consequence, \cref{C-DualConfig} tells us that $\CF(\cM)$  determines the supercloud and
subflock polynomials.
Furthermore, we obtain a second sum-product description of the Whitney function. 
Indeed, with the notation of \cref{R-CloudFlock3} we have
\[
   R_{\cM^*}=\sum_{Z^\perp\in\cZ(\cM^*)}c_{\cM^*,Z^\perp}*f_{\cM^*,Z^\perp}=\sum_{Z\in\cZ(\cM)}f^{\downarrow}_{\cM,Z}(x)*c^{\uparrow}_{\cM,Z}(y),
\]
and thus $R_{\cM}=\sum_{Z\in\cZ(\cM)}f^{\downarrow}_{\cM,Z}*c^{\uparrow}_{\cM,Z}$ thanks to \cref{T-WhitDual}.
\end{rem}

\section{From Matroid Configurations to $q$-Matroids}\label{S-ConfigMatqMat}
In this short section we will show first that every finite lattice arises as the lattice of cyclic flats of a \qM{}. 
This generalizes the analogous result for matroids, and in fact, the proof will make use of that result.
Second, we will show that every configuration arising for matroids also arises for \qM{}s, whereas the converse is not true.
For the definition of cyclic flats and configurations of matroids we refer to \cite{BoDM08,Ebe14}.
They are clearly defined analogously to \qM{}s.

\begin{theo}\label{T-Lattice}\
\begin{alphalist}
\item For every finite field~$\F$ and every finite lattice $(\cZ,\leq,\wedge,\vee)$ there exists a \qM{} over~$\F$ 
      whose lattice of cyclic flats is isomorphic to~$\cZ$.
\item Let $M=(S,r)$ be a matroid and $\Config(M)$ be its configuration. 
      For every finite field~$\F$ there exists a \qM{} $\cM=(E,\rho)$ over~$\F$ such that $\Config(\cM)=\Config(M)$ (hence $\dim E=|S|$).
\end{alphalist}
\end{theo}

For the proof we need the following simple result.

\begin{lemma}\label{L-LattIso}
Let $(\cZ,\leq,\wedge,\vee)$ be a lattice and $(\cZ',\leq')$ be a poset. Let $\sigma:\cZ\longrightarrow\cZ'$ be a bijection such that 
\[
   A\leq B\Longleftrightarrow \sigma(A)\leq'\sigma(B)
\]
for all $A,B\in\cZ$. Then $(\cZ',\leq',\wedge,\vee)$ is a lattice with meet $\sigma(A)\wedge\sigma(B)=\sigma(A\wedge B)$ and join
$\sigma(A)\vee\sigma(B)=\sigma(A\vee B)$.
Thus,~$\sigma$ is a lattice isomorphism. 
\end{lemma}

\begin{proof}
First off, $A,\,B\leq A\vee B$ implies $\sigma(A),\,\sigma(B)\leq' \sigma(A\vee B)$. Hence $\sigma(A\vee B)$ is an upper bound of 
$\sigma(A)$ and $\sigma(B)$. 
Suppose now we have another upper bound, say $\sigma(A),\,\sigma(B)\leq'\sigma(C)$ for some $C\in \cZ$.
Then $A,\,B\leq C$ and thus $A\vee B\leq C$ and $\sigma(A\vee B)\leq'\sigma(C)$.
This shows that $ \sigma(A\vee B)$ is the join of $\sigma(A)$ and $\sigma(B)$.
In the same way, one establishes $\sigma(A\wedge B)$ as the meet of $\sigma(A)$ and $\sigma(B)$.
\end{proof}

\textit{Proof of \cref{T-Lattice}:}
(a) Let $(\cZ,\leq,\wedge,\vee)$ be a finite lattice. Thanks to \cite[Thm.~2.1]{BoDM08} there exists a matroid, 
say $M=([n],r)$, such that $\cZ(M)\cong\cZ$. Without loss of generality we may assume that $\cZ(M)=\cZ$.
We proceed in several steps.
\\ 
1) Choose an $n$-dimensional $\F$-vector space $E=\subspace{e_1,\ldots,e_n}$ and consider the injective map
\begin{equation}\label{e-sigma}
  \sigma:\cL([n])\longrightarrow\cL(E),\quad \{i_1,\ldots,i_t\}\longmapsto\subspace{e_{i_1},\ldots,e_{i_t}}.
\end{equation}
Then~$\sigma$ is inclusion-preserving and satisfies $|A|=\dim\sigma(A)$ for all $A\subseteq[n]$. Moreover,
\begin{equation}\label{e-InterSum}
    \sigma(A\cap B)=\sigma(A)\cap\sigma(B)\ \text{ and }\ \sigma(A\cup B)=\sigma(A)+\sigma(B).
\end{equation}
Consider the poset $(\cZ',\leq)$, where $\cZ'=\{\sigma(A)\mid A\in\cZ\}$ and $\leq$ denotes subspace containment. 
Thus~$\sigma$ induces a bijection between the lattice $(\cZ,\subseteq)$ and $(\cZ',\leq)$.
Furthermore,~$\sigma$ satisfies
$A\subseteq B\Longleftrightarrow \sigma(A)\leq\sigma(B)$ for all $A,B\in\cZ$, and therefore \cref{L-LattIso} implies that $\cZ'$ is a lattice and~$\sigma$ a lattice isomorphism, that is,
\begin{equation}\label{e-LattIso}
  \sigma(A\wedge B)=\sigma(A)\wedge\sigma(B)\ \text{ and }\  \sigma(A\vee B)=\sigma(A)\vee \sigma(B).
\end{equation}
2) Define the map $\rho:\cZ'\longrightarrow\N_0$ via
\begin{equation}\label{e-rhosigma}
  \rho(\sigma(A))=r(A)\text{ for all }A\in\cZ.
\end{equation}
We show next that all $\sigma(A),\sigma(B)\in\cZ'$ satisfy the following properties.
\\[.5ex]
(Z0) $\sigma(A)\wedge\sigma(B)\leq \sigma(A)\cap\sigma(B)$ and $\sigma(A)+\sigma(B)\leq\sigma(A)\vee\sigma(B)$;\\
(Z1) $\sigma(0_{\cZ})=0_{\widetilde{\cZ}}$ and $\rho(0_{\widetilde{\cZ}})=0$;\\
(Z2) $0<\rho(\sigma(A))-\rho(\sigma(B))<\dim\sigma(A)-\dim\sigma(B)$ whenever $\sigma(B)\lneq\sigma(A)$;\\
(Z3) $\rho(\sigma(A))+\rho(\sigma(B))\geq \rho(\sigma(A)\vee\sigma(B))+\rho(\sigma(A)\wedge\sigma(B))+
        \dim\big((\sigma(A)\cap\sigma(B))/(\sigma(A)\wedge\sigma(B))\big)$.\\[.5ex]
(Z0) follows from $A\wedge B\subseteq A\cap B$ and $A\cup B\subseteq A\vee B$ together with \eqref{e-InterSum} and \eqref{e-LattIso}.
The remaining statements follow from the according properties of cyclic flats of matroids given in \cite[Thm.~3.2]{BoDM08}.
Precisely,~(Z1) is clear with \cite[Thm.~3.2(Z1)]{BoDM08}.
For~(Z2) note that $\sigma(B)\lneq\sigma(A)$ implies $B\subsetneq A$. 
By definition of~$\rho$, the inequalities follow immediately from \cite[Thm.~3.2(Z2)]{BoDM08}.
In the same way, (Z3) follows from \eqref{e-LattIso} and \cite[Thm.~3.2(Z3)]{BoDM08} together with the identity
\begin{align*}
  \dim\big((\sigma(A)\cap\sigma(B))/(\sigma(A)\wedge\sigma(B))\big)&=\dim(\sigma(A\cap B))-\dim(\sigma(A\wedge B))\\
   &=|(A\cap B)-(A\wedge B)|.
\end{align*}
3) The previous part shows that $(\cZ',\leq,\wedge,\vee)$ together with the map $\rho:\cZ'\longrightarrow\N_0$ 
satisfies the criterion for cyclic flats of \qM{}s given in \cite[Def.~4.1]{AlBy24}, and therefore \cite[Cor.~4.12]{AlBy24} tells us that 
the map $\rho:\cZ'\longrightarrow\N_0$ extends to a map~$\rho:\cL(E)\longrightarrow\N_0$ such that  $\cM=(E,\rho)$ is a \qM{} and 
$\cZ(\cM)=\cZ'$.
This concludes the proof of~(a).
\\
(b) Given $M=(S,r)$ and $\Config(M)$. WLOG we may assume that $S=[n]$. 
Consider the \qM{} $\cM=(E,\rho)$ as constructed in the proof of~(a). 
Then $\dim\sigma(A)=|A|$ and $\rho(\sigma(A))=r(A)$ for all $A\in\cZ(M)$; see \eqref{e-sigma} and \eqref{e-rhosigma}. 
This shows that the lattice isomorphism $\sigma:\cZ(M)\longrightarrow\cZ(\cM)$ preserves the corank and nullity of the cyclic sets, and thus $\Config(M)=\Config(\cM)$. \hfill $\square$

The process of associating a \qM{} to a given matroid as in the above proof may be regarded as a converse of \cite[Thm.~4.4]{GLJ22Gen}.

The converse of \cref{T-Lattice}(b) is not true, that is, there exist \qM{}s whose configurations do not arise as the configuration of a matroid. 

\begin{exa}\label{E-NoConfig}
Consider again Examples~\ref{E-CloudFlock2} and~\ref{E-ConfigCF}.
The configuration and lattice of cyclic flats are given by 
\begin{equation}\label{e-ZLattice}
\begin{array}{l}
   \begin{xy}
   (25,0)*+{\mbox{$(3,0)$}}="c0";%
   (10,10)*+{\mbox{$(2,1)$}}="c1";%
   (20,10)*+{\mbox{$(2,1)$}}="c2";%
   (30,10)*+{\mbox{$(1,1)$}}="c3";%
   (40,10)*+{\mbox{$(1,1)$}}="c4";%
   (50,10)*+{\mbox{$(1,1)$}}="c5";%
   (15,20)*+{\mbox{$(1,2)$}}="c6";%
   (25,30)*+{\mbox{$(0,3)$}}="c7";%
    {\ar@{-}"c0";"c1"};
    {\ar@{-}"c0";"c2"};
    {\ar@{-}"c0";"c3"};
    {\ar@{-}"c0";"c4"};
    {\ar@{-}"c0";"c5"};
    {\ar@{-}"c1";"c6"};
    {\ar@{-}"c2";"c6"};
    {\ar@{-}"c3";"c7"};
    {\ar@{-}"c4";"c7"};
    {\ar@{-}"c5";"c7"};    
    {\ar@{-}"c6";"c7"};
   \end{xy}
\end{array}
\qquad
\begin{array}{l}
   \begin{xy}
   (25,0)*+{\emptyset}="c0";%
   (10,10)*+{Z_1}="c1";%
   (20,10)*+{Z_2}="c2";%
   (30,10)*+{Z_3}="c3";%
   (40,10)*+{Z_4}="c4";%
   (50,10)*+{Z_5}="c5";%
   (15,20)*+{Z_6}="c6";
   (25,30)*+{Z_7}="c7";%
    {\ar@{-}"c0";"c1"};
    {\ar@{-}"c0";"c2"};
    {\ar@{-}"c0";"c3"};
    {\ar@{-}"c0";"c4"};
    {\ar@{-}"c0";"c5"};
    {\ar@{-}"c1";"c6"};
    {\ar@{-}"c2";"c6"};
    {\ar@{-}"c3";"c7"};
    {\ar@{-}"c4";"c7"};
    {\ar@{-}"c5";"c7"};    
    {\ar@{-}"c6";"c7"};
   \end{xy}
\end{array}
\end{equation}
Suppose there exists a matroid $M=([n],r)$ such that $\Config(M)$ equals the above configuration.
The top and bottom corank-nullity pairs tell us that $r([n])=3$ and $n=6$.
Thus, the lattice of cyclic flats of~$M$ has the above form where $Z_7=[6]$.
The corank-nullity data imply
\[
  (r(Z_i),|Z_i|)=\left\{\begin{array}{cl}(1,2),&\text{for }i=1,2,\\ (2,3),&\text{for }i=3,4,5,\\(2,4),&\text{for }i=6,\\(3,6),&\text{for }i=7.\end{array}\right.
\]
Recall that \cite[Thm.~3.2(Z3)]{BoDM08} tells us that for all $i,j\in\{1,\ldots,7\}$
\[
   r(Z_i)+r(Z_j)\geq r(Z_i\vee Z_j)+r(Z_i\wedge Z_j)+|(Z_i\cap Z_j)-(Z_i\wedge Z_j)|.
\]
For $(i,j)=(1,2)$ this yields $1+1\geq2+0+|Z_1\cap Z_2|$, and thus $Z_1\cap Z_2=\emptyset$. 
In the same way we obtain $Z_1\cap Z_3=\emptyset=Z_2\cap Z_3$.
But there are no three pairwise disjoint subsets of cardinality $2,2$, and~$3$ in the set~$[6]$.
Hence there exists no matroid~$M$ whose configuration equals the one given in \cref{E-ConfigCF}.
Note that there exist matroids whose lattice of cyclic flats (disregarding corank-nullity data) equals the one in \eqref{e-ZLattice}; see also 
\cite[Thm.~2.1]{BoDM08}.
For instance, choose $Z_7=[8]$ of rank~$4$ and 
\[
   Z_6\!=\!\{1,2,3,4,5\}, Z_1\!=\!\{1,2,3\}, Z_2\!=\!\{3,4,5\}, Z_3\!=\!\{1,4,6,7\}, Z_4\!=\!\{1,5,6,8\}, Z_5\!=\!\{2,4,6,8\}
\]
with rank values $3,\,2,\,2,\,3,\,3,\,3$ in the given order. 
Then the properties (Z0)--(Z4) in \cite[Thm.~3.2]{BoDM08} are satisfied and thus there exists a matroid $M=([8],r)$ whose lattice of cyclic flats is the 
one in \eqref{e-ZLattice}.
\end{exa}

\section{Direct Sums}\label{S-DirSums}
Our first goal is to generalize \cref{T-ConfigWhitney} to arbitrary \qM{}s.
To do so, we need to consider direct sums of the form~$\cM=\cM_1\oplus \cU$, where $\cU$ is either free or trivial.
We will show that the configuration, cloud-flock lattice and Whitney function of~$\cM$ are fully determined by the according invariants of~$\cM_1$.
The proofs are somewhat lengthy, but very straightforward. 
As to our knowledge these rather unsurprising facts have not been worked out elsewhere.
Thereafter we will discuss duality and general direct sums. 

We start with showing that the configurations of two \qM{}s~$\cM_1$ and~$\cM_2$ determine the configuration of the direct sum $\cM_1\oplus\cM_2$.

To do so, recall that the direct product of posets $(\cP,\leq),\,(\cQ,\leq)$ is defined as $(\cP\times\cQ,\,\leq\,)$, 
where~$\cP\times\cQ$ is the cartesian product and the partial order is given by 
\[
   (a,b)\leq (a',b')\text{ if $a\leq a'$ and $b\leq b'$.}
\]
If $\cP$ and $\cQ$ are lattices, the same is true for
$\cP\times\cQ$ and $(a,b)\wedge(a',b')=(a\wedge a',b\wedge b')$ and $(a,b)\vee(a',b')=(a\vee a',b\vee b')$.

\begin{prop}\label{P-DirSumConfig}
Let $\cM_i=(E_i,\rho_i)$ for $i=1,2$ and $\cM=\cM_1\oplus\cM_2=(E,\rho)$ be the direct sum, where $E=E_1\oplus E_2$. 
As in \cref{D-Config}(a) let $\lambda_i(Z_i)=(\rho_i(E_i)-\rho_i(Z_i),\,\dim Z_i-\rho_i(Z_i))$ for $Z_i\in\cZ(\cM_i)$ and similarly let
$\lambda(Z)=(\rho(E)-\rho(Z),\,\dim Z-\rho(Z))$ for $Z\in\cZ(\cM)$.
Then the map
\[
             \Config(\cM_1)\times\Config(\cM_2)\longrightarrow\Config(\cM) ,\quad
            \big(\lambda_1(Z_1),\lambda_2(Z_2)\big)\longmapsto \lambda_1(Z_1)+\lambda_2(Z_2)
\]
is an isomorphism of lattices.
\end{prop}

\begin{proof}
From \cref{T-DSCycFlats} we know that $\cZ(\cM)=\{Z_1\oplus Z_2\mid Z_i\in\cZ(\cM_i)\}$. 
The subspaces in $\cZ(\cM)$ clearly satisfy 
$Z_1\oplus Z_2\leq Z_1'\oplus Z_2'\Longleftrightarrow Z_i\leq Z_i'$ for $i=1,2$, and therefore 
we have a lattice isomorphism
$\cZ(\cM_1)\times\cZ(\cM_2)\longrightarrow\cZ(\cM),\ (Z_1,Z_2)\longmapsto Z_1\oplus Z_2$.
Furthermore, by \eqref{e-RhoDirect}
\begin{align*}
    \lambda_1(Z_1)+\lambda_2(Z_2)&=\big(\rho_1(E_1)+\rho_2(E_2)-\rho_1(Z_1)-\rho_2(Z_2),\,\dim Z_1+\dim Z_2-\rho_1(Z_1)-\rho_2(Z_2)\big)\\
        &=\big(\rho(E)-\rho(Z_1\oplus Z_2),\,\dim(Z_1\oplus Z_2)-\rho(Z_1\oplus Z_2)\big)=\lambda(Z_1\oplus Z_2).
\end{align*}
Recalling the definition of the configuration lattices in \cref{D-Config}(a), we conclude that the given map is a lattice isomorphism, as stated.
\end{proof}

In order to discuss direct sums involving a free or trivial summand, we need the following results from basic Linear Algebra, whose proof is 
in the appendix.

\begin{lemma}\label{L-ExtendSpaces}
For $i=1,2$ let $\dim E_i=n_i$ and $V_i\leq E_i$ such that $\dim V_i=k_i$. Let $j\in[n_1+n_2]$.
\begin{alphalist}
\item Set $\cV=\{V\leq E\mid \pi_1(V)=V_1,\, \dim V=j\}$. 
         Then 
         \[
             \big|\cV\big|=\left\{\begin{array}{cl}
                \Gaussian{n_2}{j-k_1}_q q^{k_1(n_2-j+k_1)},&\text{if }j\in\{k_1,\ldots,k_1+n_2\},\\[.7ex]
                  0,&\text{otherwise.}\end{array}\right.
         \]

\item Set $\cV'=\{V\leq E\mid V\cap E_2=V_2\text{ and }\dim \pi_1(V)=j\}$. Then
         \[
             \big|\cV'\big|=\left\{\begin{array}{cl}
                \Gaussian{n_1}{j}_q q^{j(n_2-k_2)},&\text{if }j\in\{0,\ldots,n_1\},\\[.7ex]
                  0,&\text{otherwise.}\end{array}\right.
         \]
         Note that all spaces in $\cV'$ have dimension~$j+k_2$.
\end{alphalist}
\end{lemma}

The next lemma about the flats will help us determine the clouds of direct sums involving a free or trivial summand.
Its proof has been deferred to the appendix.

\begin{lemma}\label{L-FlatsFree}
Let $\cM_1=(E_1,\rho_1)$ be a \qM{} on an $n_1$-dimensional ground space~$E_1$.
\begin{alphalist}
\item Let $\cM=\cM_1\oplus\cU_{n_2,n_2}(E_2)$. Set $E=E_1\oplus E_2$. Then any $V\in\cL(E)$ satisfies
      \begin{equation}\label{e-DimRhoFree}
           \dim V=\dim(V\cap E_1)+\dim\pi_2(V)\ \; \text{ and }\ \; \rho(V)=\rho_1(V\cap E_1)+\dim\pi_2(V).
      \end{equation}
      Furthermore, $\cF(\cM)=\{F\in\cL(E)\mid F\cap E_1=\cF(\cM_1)\}$.
\item Let $\cM=\cM_1\oplus\cU_{0,n_2}(E_2)$. Set $E=E_1\oplus E_2$. Then any $V\in\cL(E)$ satisfies
      \begin{equation}\label{e-DimRhoTriv}
           \rho(V)=\rho_1(\pi_1(V)).
      \end{equation}
      Furthermore, $\cF(\cM)=\{F_1\oplus E_2\mid F_1\in\cF(\cM_1)\}$.
\end{alphalist}
\end{lemma}

Now we are ready for the particular direct sums. We start with those involving a free \qM.

\begin{prop}\label{P-FreeSummand}
Let $\cM_1=(E_1,\rho_1)$ be a \qM{} on an $n_1$-dimensional ground space~$E_1$ with rank function~$\rho_1$ and $\rho_1(\cM_1)=\hat{\rho}$.
Let $\cM=\cM_1\oplus\cU_{n_2,n_2}(E_2)$. 
Then 
\begin{alphalist}
\item $\cZ(\cM)=\{Z_1\oplus0\mid Z_1\in\cZ(\cM_1)\}$, and thus the lattices $\cZ(\cM)$ and~$\cZ(\cM_1)$ are isomorphic. 
        Furthermore,  $\Config(\cM_1)$ is isomorphic to $\Config(\cM)$ via the map
        $(a,b)\longmapsto(a+n_2,b)$.
\item $\cl^{-1}(Z_1\oplus0)=\cl^{-1}(Z_1)\oplus0$ and $f_{\cM,Z_1\oplus 0}=f_{\cM_1,Z_1}$ for all $Z_1\in\cZ(\cM_1)$.
\item Let $Z_1\in\cZ(\cM_1)$. Then $\cyc^{-1}(Z_1\oplus0)=\{F\leq E_1\oplus E_2\mid F\cap E_1\in\cyc^{-1}(Z_1)\}$.
        Moreover, if $\rho_1(Z_1)=r$ and $\dim Z_1=d$, then 
        \[
               c_{\cM,Z_1\oplus0}=\sum_{t=r}^{\hat{\rho}+n_2}\sum_{j=r}^{t}b_j
                        \GaussianD{n_2}{t-j}_q q^{(t-j)(n_1-j-d+r)}x^{\hat{\rho}+n_2-t},         
               \text{ where }
               c_{\cM_1,Z_1}=\sum_{j=r}^{\hat{\rho}}b_jx^{\hat{\rho}-j}.
         \]
         As a consequence, $\CF(\cM)$ and $\CF(\cM_1)$ fully determine each other, and 
         $(c_{\cM_1,Z_1},f_{\cM_1,Z_1})\rightarrow(c_{\cM,Z_1\oplus0},f_{\cM_1,Z_1})$ is a lattice isomorphism.
\item Let $R_{\cM_1}=\sum_{i=0}^{\hat{\rho}}\sum_{j=0}^{n_1-\hat{\rho}}\nu_{i,j}^{(1)}x^iy^j$. Then 
        \[
           R_{\cM}=\sum_{\ell=0}^{\hat{\rho}}\sum_{b=0}^{n_1-\hat{\rho}}\nu_{\ell,b}^{(1)} \sum_{a=\ell}^{\hat{\rho}+n_2}
          \GaussianD{n_2}{a-\ell}_q q^{(\ell-a+n_2)(n_1-\hat{\rho}-b+\ell)}x^ay^b.
        \]
\end{alphalist}
\end{prop}

\begin{proof}
Set $E=E_1\oplus E_2$ and let $\rho$ be the rank function of~$\cM$. 
Then $\rho$ satisfies \eqref{e-DimRhoFree} and  $\rho(\cM)=\hat{\rho}+n_2$.
\\
(a) The first part follows from \cref{T-DSCycFlats} and \cref{E-CloudFlockUniform} and the second part is a special case of 
\cref{P-DirSumConfig}.
\\
(b) Let $Z=Z_1\oplus0\in\cZ(\cM)$. Then \cref{P-CloudFlockProp} and \eqref{e-DimRhoFree} imply $\cl^{-1}(Z)=\{V\leq E_1\oplus E_2\mid V\leq Z,\, \rho(V)=\rho(Z)\}=\cl^{-1}(Z_1)\oplus0$.
Together with~(a) we conclude $f_{\cM,Z}=f_{\cM_1,Z_1}$.
\\
(c) Let $F\in\cL(E)$ and set $\hat{F}:=F\cap E_1$. Then \cref{P-CloudFlockProp} and \cref{L-FlatsFree}(a) imply
\begin{align*}
  F\in\cyc^{-1}(Z_1\oplus0)&\Longleftrightarrow F\in\cF(\cM),\, Z_1\leq \hat{F},\, \dim\hat{F}-\rho_1(\hat{F})=\dim Z_1-\rho_1(Z_1)\\
    &\Longleftrightarrow \hat{F}\in\cyc^{-1}(Z_1).
\end{align*}
This proves the first part.
Let now $Z_1\in\cZ(\cM_1)$ and $\rho_1(Z_1)=r$ and $\dim Z_1=d$. Set $Z=Z_1\oplus0$. Let $c_{\cM_1,Z_1}$ be as given.
By \cref{P-CloudFlockProp}  any subspace $\hat{V}\in\cyc^{-1}(Z_1)$ such that $\rho_1(\hat{V})=j$ satisfies $\dim\hat{V}=j+d-r$.
Thus, choosing any subspace  $\hat{V}\in E_1$ such that $\dim\hat{V}=j+d-r$, we compute
\begin{align*}
  c_t&:=\big|\{F\in\cyc^{-1}(Z)\mid \rho(F)=t\}\big|= \sum_{j=r}^{t}\big|\{F\in\cyc^{-1}(Z)\mid \rho(F)=t,\,\rho_1(F\cap E_1)=j\}\big|\\
  &=\sum_{j=r}^{t}b_j\big|\{F\in\cL(E)\mid \rho(F)=t,\,F\cap E_1=\hat{V},\,\dim\pi_2(F)=t-j\}\big|\\
  &=\sum_{j=r}^{t}b_j\GaussianD{n_2}{t-j}_q q^{(t-j)(n_1-j-d+r)},
\end{align*}
where the second step follows from the fact that $F\in\cyc^{-1}(Z)$ with $\rho(F)=t$ and $\dim(F\cap E_1)=j+d-r$
satisfies $\dim F=t+d-r$ (see \cref{P-CloudFlockProp}) and thus $\dim\pi_2(F)=t-j$,
and the last step is a consequence of \cref{L-ExtendSpaces}(b).
Recalling the degree of the cloud polynomial from \cref{R-DegCloudFlock},  we arrive at
\[
  c_{\cM,Z}=\sum_{t=r}^{\hat{\rho}+n_2} c_tx^{\hat{\rho}+n_2-t}
  =\sum_{t=r}^{\hat{\rho}+n_2}\sum_{j=r}^{t}b_j\GaussianD{n_2}{t-j}_q q^{(t-j)(n_1-j-d+r)}x^{\hat{\rho}+n_2-t},
\]
as desired.
Finally, the above shows that $(b_r,\ldots,b_{\hat{\rho}})$ and $(c_r,\ldots,c_{\hat{\rho}})$ are related via an invertible triangular matrix (and $(c_{\hat{\rho}+1},\ldots,c_{\hat{\rho}+n_2})$ are determined by $(c_r,\ldots,c_{\hat{\rho}})$).
Thus $c_{\cM_1,Z_1}$ and $c_{\cM,Z}$ fully determine each other, and 
this establishes the remaining statements of~(c).
\\
(d) 
From \cref{R-WhitneyCounting} we know that 
$\nu_{i,j}^{(1)}=\big|\{V_1\leq E_1\mid \rho_1(V_1)=\hat{\rho}-i,\,\dim V_1=\hat{\rho}-i+j\}\big|$.
Let $a\in\{0,\ldots,\hat{\rho}+n_2\},\, b\in\{0,\ldots,n_1-\hat{\rho}\}$.
Picking for any $j\in\{n_2-a,\ldots,n_2-a+\hat{\rho}\}$ a subspace $\hat{V}_{1,j}\leq E_1$ such that 
$\dim \hat{V}_{1,j}=n_2+\hat{\rho}-a-j+b$, we compute with the aid of \eqref{e-WhitCoeff} and \eqref{e-DimRhoFree}
\begin{align*}
   \nu_{a,b}&=\big|\{V\leq E\mid \rho(V)=\hat{\rho}+n_2-a,\,\dim V=\hat{\rho}+n_2-a+b\}\big|\\
      &=\sum_{j=n_2-a}^{n_2-a+\hat{\rho}}\big|\{V\leq E\mid \dim V=\hat{\rho}+n_2-a+b,\,\dim\pi_2(V)=j,\,
                \rho_1(V\cap E_1)=\hat{\rho}+n_2-a-j\}\big|\\
      &=\sum_{j=n_2-a}^{n_2-a+\hat{\rho}}\nu_{a-n_2+j,b}^{(1)}\big|\{V\leq E\mid \dim V=\hat{\rho}+n_2-a+b,\,\dim\pi_2(V)=j,\,
           V\cap E_1=\hat{V}_{1,j}\}\big|\\
      &=\sum_{j=n_2-a}^{n_2-a+\hat{\rho}}\nu_{a-n_2+j,b}^{(1)}\GaussianD{n_2}{j}_qq^{j(n_1-n_2-\hat{\rho}+a+j-b)},
\end{align*}
where the last step follows from \cref{L-ExtendSpaces}(b). Now we obtain
\begin{align*}
  R_\cM&=\sum_{a=0}^{\hat{\rho}+n_2}\sum_{b=0}^{n_1-\hat{\rho}}\sum_{j=n_2-a}^{n_2-a+\hat{\rho}}
             \nu_{a-n_2+j,b}^{(1)}\GaussianD{n_2}{j}_qq^{j(n_1-n_2-\hat{\rho}+a+j-b)}x^ay^b\\
     &=\sum_{\ell=0}^{\hat{\rho}}\sum_{b=0}^{n_1-\hat{\rho}}\nu_{\ell,b}^{(1)} \sum_{a=0}^{\hat{\rho}+n_2}
          \GaussianD{n_2}{a-\ell}_q q^{(\ell-a+n_2)(n_1-\hat{\rho}-b+\ell)}x^ay^b. \qedhere
\end{align*}
\end{proof}

It remains to discuss direct sums of the form $\cM=\cM_1\oplus\cU_{0,n_2}(E_2)$. 
By \cite[Thm.~5.8]{GLJ24DSCyc} we have $\cM^*=\cM_1^*\oplus\cU_{n_2,n_2}(E_2)$, which is of the form as in \cref{P-FreeSummand}.
However, since our duality result in \cref{C-DualConfig}(b) only applies to full \qM{}s,  this case needs to be worked out directly (for the most part).

\begin{prop}\label{P-TrivialSummand}
Let $\cM_1=(E_1,\rho_1)$ be a \qM{} on an $n_1$-dimensional ground space~$E_1$ with rank function~$\rho_1$ and $\rho_1(E_1)=\hat{\rho}$.
Let $\cM=\cM_1\oplus\cU_{0,n_2}(E_2)$. 
Then
\begin{alphalist}
\item $\cZ(\cM)=\{Z_1\oplus E_2\mid Z_1\in\cZ(\cM_1)$, and thus the lattices $\cZ(\cM)$ and~$\cZ(\cM_1)$ are isomorphic. 
        Furthermore, $\Config(\cM_1)$ is isomorphic to $\Config(\cM)$ via the map $(a,b)\longmapsto(a,b+n_2)$. 
\item Let $Z_1\in\cZ(\cM_1)$. Then $\cl^{-1}(Z_1\oplus E_2)=|\{V\in\cL(E)\mid \pi_1(V)\in\cl^{-1}(Z_1)\}$. 
     Moreover, if $\dim Z_1=d$ and $\rho_1(Z_1)=r$, then 
     \[
       f_{\cM,Z_1\oplus E_2}=\sum_{j=r}^{d+n_2}\sum_{t=r}^j b_t\GaussianD{n_2}{j-t}_q q^{t(n_2-j+t)} y^{j-r},\ \text{ where }
       f_{\cM_1,Z_1}=\sum_{t=r}^d b_t y^{t-r}.
     \]
\item Let $Z_1\in\cZ(\cM_1)$. Then $\cyc^{-1}(Z_1\oplus E_2)=\cyc^{-1}(Z_1)\oplus E_2$ and 
        $c_{\cM,Z_1\oplus E_2}=c_{\cM_1,Z_1}$.
        As a consequence, $\CF(\cM)$ and $\CF(\cM_1)$ determine each other, and the map
         $(c_{\cM_1,Z_1},f_{\cM_1,Z_1})\rightarrow(c_{\cM_1,Z_1},f_{\cM,Z_1\oplus E_2})$ is a lattice isomorphism.
\item  Let $n=n_1+n_2$ and $R_{\cM_1}=\sum_{i=0}^{\hat{\rho}}\sum_{j=0}^{n_1-\hat{\rho}}\nu_{i,j}^{(1)}x^iy^j$. Then 
         \[
           R_{\cM}=\sum_{a=0}^{\hat{\rho}}\sum_{j=0}^{n_1-\hat{\rho}}\nu_{a,j}^{(1)}
                \sum_{b=j}^{n-\hat{\rho}}\GaussianD{n_2}{b-j}_qq^{(\hat{\rho}-a+j)(n_2-b+j)}x^ay^b.
        \]
\end{alphalist}
\end{prop}

\begin{proof}
Let~$\rho$ be the rank function of~$\cM$ and set $E=E_1\oplus E_2$. Then $\rho(\cM)=\hat{\rho}$ and $\rho(V)=\rho_1(\pi_1(V))$ for all $V\in\cL(E)$;
see \cref{L-FlatsFree}(b).
\\
(a) follows from \cref{T-DSCycFlats}, \cref{E-CloudFlockUniform} and \cref{P-DirSumConfig}.
\\
(b) Let $Z_1\in\cZ(\cM_1)$ and set $Z=Z_1\oplus E_2$. 
Let $V\in\cL(E)$. Then 
\[
  V\in\cl^{-1}(Z)\Longleftrightarrow V\leq Z_1\oplus E_2\ \text{ and }\ \rho(V)=\rho_1(\pi_1(V))=\rho_1(Z_1)
  \Longleftrightarrow \pi_1(V)\in\cl^{-1}(Z_1),
\]
and this proves the statement about $\cl^{-1}(Z)$.
Let now $\dim Z_1=d$ and $\rho_1(Z_1)=r$, and let $f_{\cM_1,Z_1}$ be as in the proposition.
Then $b_t=\big|\{W\in\cl^{-1}(Z_1)\mid \dim W=t\}\big|$.
For any  subspace $W\leq E_1$ of dimension~$t$, \cref{L-ExtendSpaces}(a) tells us that 
$\big|\{V\leq E\mid \dim V=j,\,\pi_1(V)=W\}\big|=\Gaussian{n_2}{j-t}_q q^{t(n_2-j+t)}$ for all $j\in\{t,\ldots,t+n_2\}$.
Hence
\[
  \big|\{V\in\cl^{-1}(Z)\mid \dim V=j\}\big|=\sum_{t=r}^j b_t \GaussianD{n_2}{j-t}_q q^{t(n_2-j+t)},
\]
and this establishes the expression for the flock polynomial.
As in the proof of \cref{P-FreeSummand}(c) we observe that $f_{\cM_1,Z_1}$ and $f_{\cM,Z_1\oplus E_2}$ determine each other.
\\
(c) Set $Z=Z_1\oplus E_2$. Let $F\in\cL(E)$. With the aid of \cref{L-FlatsFree}(b) we obtain
\pagebreak
\begin{align*}
   F\in\cyc^{-1}(Z)&\Longleftrightarrow F\in\cF(\cM),\,Z\leq F,\,\dim F-\rho(F)=\dim Z-\rho(Z)\\ 
       &\Longleftrightarrow F=F_1\oplus E_2,\\
       &\mbox{}\hspace*{2.5em} \text{where $F_1\in\cF(\cM_1),\,Z_1\leq F_1$, and }\dim F_1-\rho_1(F_1)=\dim Z_1-\rho_1(Z_1)\\
      &\Longleftrightarrow F=F_1\oplus E_2\text{ for some }F_1\in\cyc^{-1}(Z_1).
\end{align*}
This shows the identity about $\cyc^{-1}(Z)$. 
As a consequence, $\big|\{F\in\cyc^{-1}(Z)\mid \rho(F)=j\}\big|=|\big\{F_1\in\cyc^{-1}(Z_1)\mid \rho_1(F_1)=j\}\big|$ and thus 
$c_{\cM,Z}=c_{\cM_1,Z_1}$.
Again, the statement about $\CF(\cM)$ is clear.
\\
(d) $\cM=\cM_1\oplus\cU_{0,n_2}(E_2)$ implies $\cM^*=\cM_1^*\oplus\cU_{n_2,n_2}(E_2)$; see \cite[Thm.~5.8]{GLJ24DSCyc}.
Furthermore, $\rho^*(\cM_1^*)=\tilde{\rho}:=n_1-\hat{\rho}$ and 
$R_{\cM_1^*}=\sum_{\ell=0}^{\tilde{\rho}}\sum_{a=0}^{n_1-\tilde{\rho}}\nu_{a,\ell}^{(1)}x^\ell y^a$
by \cref{T-WhitDual}.
Now \cref{P-FreeSummand}(d) implies
\begin{align*}
   R_{\cM^*}&=\sum_{\ell=0}^{\tilde{\rho}}\sum_{a=0}^{n_1-\tilde{\rho}}\nu_{a,\ell}^{(1)}\sum_{b=\ell}^{\tilde{\rho}+n_2}\GaussianD{n_2}{b-\ell}_q
                            q^{(\ell-b+n_2)(n_1-\tilde{\rho}-a+\ell)}x^by^a\\
                  &=\sum_{a=0}^{\hat{\rho}}\sum_{\ell=0}^{n_1-\hat{\rho}}\nu_{a,\ell}^{(1)}\sum_{b=\ell}^{n-\hat{\rho}}\GaussianD{n_2}{b-\ell}_q
                            q^{(\ell-b+n_2)(\hat{\rho}-a+\ell)}x^by^a.
\end{align*}
Now $R_{\cM}(x,y)=R_{\cM^*}(y,x)$ leads to the stated identity.
\end{proof}

\begin{exa}\label{E-PrimeFree}
In \cite{ByFu25}, the authors determine the Whitney function for prime-free \qM{}s. 
In \cite[Rem.~7.12]{GLJ24DSCyc} it has been shown that these are exactly the \qM{}s that form a direct sum of a free and a trivial 
\qM{}.
In that case we may apply either of Propositions~\ref{P-FreeSummand} and \ref{P-TrivialSummand}, and
together with \cref{E-WhitneyUnif},  either result leads for $\cM=\cU_{n_1,n_1}(E_1)\oplus\cU_{0,n_2}(E_2)$ to
\[
   R_\cM=\sum_{a=0}^{n_1}\sum_{b=0}^{n_2}\GaussianD{n_1}{a}_q\GaussianD{n_2}{b}_q q^{(n_1-a)(n_2-b)}x^a y^b,
\]
which coincides with \cite[Cor.~66]{ByFu25}.
\end{exa}

Now we are ready for the generalization of \cref{T-ConfigWhitney}.

\begin{theo}\label{T-ConfigWhitneyGen}
Let $\cM=(E,\rho)$ be any \qM{}.
Then $\Config(\cM)$ fully determines $\CF(\cM)$ and thus the Whitney function.
\end{theo}

\begin{proof}
By \cite[Thm.~7.9(a)]{GLJ24DSCyc} $\cM$ can be written in the form $\cM=\cM_1\oplus\cU_{0,n_2}(E_2)\oplus\cU_{n_3,n_3}(E_3)$ for suitable subspaces~$E_2,\,E_3$ of~$E$, and where $\cM_1$ is full. 
From Propositions~\ref{P-FreeSummand}(a) and~\ref{P-TrivialSummand}(a) we can determine 
$\Config(\cM_1)$ from $\Config(\cM)$.
By \cref{T-ConfigWhitney} $\Config(\cM_1)$ determines $\CF(\cM_1)$, 
which in turn 
determines $\CF(\cM)$ thanks to (b) and~(c) of Propositions~\ref{P-FreeSummand} and~\ref{P-TrivialSummand}.
Now~$R_\cM$ is determined by \cref{C-CloudFlock2}.
\end{proof}

The same proof as for \cref{C-DualConfig}(b) with \cref{T-ConfigWhitneyGen} instead of \cref{T-ConfigWhitney} immediately gives us

\begin{cor}\label{C-DualCF}
Let $\cM=(E,\rho)$ be any \qM{}.
Then $\CF(\cM)$ fully determines $\CF(\cM^*)$.
\end{cor}

The following result completes \cref{P-DirSumConfig}.

\begin{theo}\label{T-DirSumCF}
Let $\cM_1,\,\cM_2$ be \qM{}s and $\cM=\cM_1\oplus\cM_2$ be the direct sum. 
Then  $\CF(\cM)$ is fully determined by $\CF(\cM_1)$ and $\CF(\cM_2)$.
\end{theo}

\begin{proof}
By \cref{R-DegCloudFlock}, $\CF(\cM_1)$ and $\CF(\cM_2)$ determine $\Config(\cM_1)$ and $\Config(\cM_2)$, which by 
\cref{P-DirSumConfig} determine $\Config(\cM)$. 
This in turn fully determines $\CF(\cM)$ thanks to \cref{T-ConfigWhitneyGen}.
\end{proof}

In \cref{S-Converse} we will see that the above result is not true for the Whitney function, that is, $R_{\cM_1\oplus\cM_2}$ is not a function
of $R_{\cM_1}$ and $R_{\cM_2}$.

\section{Condensed Configurations}\label{S-Condensed}

In Examples \ref{E-ConfigCF} and \ref{E-CloudFlock3} we have seen that all cyclic flats with the same corank-nullity pair play the 
same role within the lattice $\cZ(\cM)$. Even more, they all have the same cloud and flock polynomial. 
As one may expect, this is not the case in general.
In this section we will have a closer look at this question.
Our discussion follows very closely the matroid case in \cite[Sec.~5]{Ebe14}, and, in fact, the results do not differ from the matroid case.

\begin{defi}\label{D-Condens}
Let $\cM=(E,\rho)$ be a \qM{}. A \textbf{condensation}~$\cP$ of $\cZ(\cM)$ is a partition of $\cZ(\cM)$ with the properties:
\begin{arabiclist}
\item The elements in every block $B$ of~$\cP$ have the same corank-nullity pair (and thus the same rank and dimension).
\item For all blocks $B,B'$ of~$\cP$ there exist an integer $\Gamma_{\cP}(B,B')\in\N_0$ such that 
      \[
            |\{Z\in B\mid Z\leq Z'\}|=\Gamma_{\cP}(B,B')\text{ for all }Z'\in B'.
      \] 
\end{arabiclist}\
\\[-2ex]
For a condensation~$\cP$ we define a partial order on its blocks via $B\leq B'$ if 
$\Gamma_{\cP}(B,B')>0$ (that is, every cyclic flat in~$B'$ contains at least one cyclic flat in~$B$).
\end{defi}

Note that every condensation $(\cP,\leq)$ is a poset with largest element $\{\cyc(E)\}$ and smallest element $\{\cl(0)\}$.
One can find examples showing that $(\cP,\leq)$ is in general not a lattice.

\begin{defi}\label{D-CondConfig}
Let $\cM=(E,\rho)$ be a \qM{} and let $\cP=(B_0,\ldots,B_t)$ be a condensation of $\cZ(\cM)$. 
For a block~$B$ of~$\cP$ let $\lambda(B)=(\rho(E)-\rho(Z),\dim Z-\rho(Z))$ for any $Z\in B$, that is, 
$\lambda(B)$ is the corank-nullity pair of any cyclic flat in~$B$.
We define $\CConfig(\cM):=(\Gamma_\cP,\,\Lambda_\cP)$, where 
\[
   \Gamma_\cP=\big(\Gamma_{\cP}(B_i,B_j)\big)_{i,j\in\{0,\ldots,t\}}\ \text{ and }\ 
   \Lambda_\cP=\big(\lambda(B_0),\ldots,\lambda(B_t)\big).
\]
$\CConfig(\cM)$ is called the \textbf{condensed configuration} of~$\cM$ with respect to~$\cP$.
\end{defi}

\begin{exa}\label{E-Condens}
\begin{alphalist}
\item The trivial partition $\cP=\big(\{Z\}\big)_{Z\in\cZ(\cM)}$ is a condensation. 
      In this case $\Gamma_\cP(Z,Z')=1$ iff $Z\leq Z'$ and zero otherwise. 
      In other words, $\Gamma_{\cP}$ is the adjacency matrix of the poset $\cZ(\cM)$.
\item Consider Examples~\ref{E-CloudFlock2}/\ref{E-ConfigCF}. 
      Then $\cP=\{B_0,\,B_1,\,B_2,\,B_3,\,B_4\}$, where 
      \[
        B_0=\{Z_0\},\,B_1=\{Z_1,Z_2\},\,B_2=\{Z_3,Z_4,Z_5\},\,B_3=\{Z_6\},\,B_4=\{Z_7\}
      \]
      is a condensation of $\cZ(\cM)$. The resulting condensed configuration is given by
      \[
         \Gamma_{\cP}=\begin{pmatrix}1&1&1&1&1\\0&1&0&2&2\\0&0&1&0&3\\0&0&0&1&1\\0&0&0&0&1\end{pmatrix},\
         \Lambda_\cP=\big((3,0),(2,1),(1,1),(1,2),(0,3)\big).
      \]
\end{alphalist}
\end{exa}

Thanks to the first example, every lattice $\cZ(\cM)$ has a condensation. 
It even has a coarsest condensation. This can be seen as follows.
 The set of all partitions~$\Pi(S)$ on a finite set~$S$ is a lattice
(see \cite[Prop.~3.3.1, p.~128]{Sta97}), and it is not hard to show that the join of two condensations in the lattice~$\Pi(\cZ(\cM))$ 
is again a condensation. 
Hence $\cZ(\cM)$ has a coarsest condensation.
In \cref{E-Condens}(b) this is exactly the given condensation~$\cP$.

The condensed configuration does not determine the configuration in general. 

\begin{exa}\label{E-CConfig}
Let $E=\F_2^8$ and consider the spaces $Z_0=0,\, Z_6=E$ and 
\begin{align*}
    &Z_1=\subspace{e_1,e_2,e_3},\ Z_2=\subspace{e_4,e_5,e_6},\ Z_3=\subspace{e_1+e_4,e_2+e_5,e_3+e_6},\\
    &Z_4=\subspace{e_1,e_2,e_3,e_7,e_8},\ Z_5=\subspace{e_4,e_5,e_6,e_7,e_8},\ Z_5'=\subspace{e_1,e_2,e_3,e_4+e_7,e_5+e_8}.
\end{align*}
Set $\cZ_1=\{Z_0,Z_1,Z_2,Z_3,Z_4,Z_5,Z_6\}$ and $\cZ_2=\{Z_0,Z_1,Z_2,Z_3,Z_4,Z_5',Z_6\}$.
Then $\cZ_i$ are lattices given by
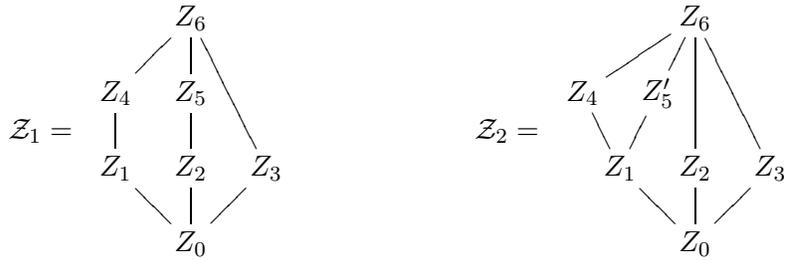
\[
\begin{array}{l}
   \begin{xy}
   (0,15)*+{\mbox{$\cZ_1=$}}="title";%
   (20,0)*+{Z_0}="c0";%
   (10,10)*+{Z_1}="c1";%
   (20,10)*+{Z_2}="c2";%
   (30,10)*+{Z_3}="c3";%
   (10,20)*+{Z_4}="c4";%
   (20,20)*+{Z_5}="c5";%
   (20,30)*+{Z_6}="c6";
    {\ar@{-}"c0";"c1"};
    {\ar@{-}"c0";"c2"};
    {\ar@{-}"c0";"c3"};
    {\ar@{-}"c1";"c4"};
    {\ar@{-}"c2";"c5"};
    {\ar@{-}"c3";"c6"};
    {\ar@{-}"c4";"c6"};
    {\ar@{-}"c5";"c6"};    
   \end{xy}
\end{array}
\hspace*{2cm}
\begin{array}{l}
   \begin{xy}
   (-5,15)*+{\mbox{$\cZ_2=$}}="title";%
   (20,0)*+{Z_0}="c0";%
   (10,10)*+{Z_1}="c1";%
   (20,10)*+{Z_2}="c2";%
   (30,10)*+{Z_3}="c3";%
   (5,20)*+{Z_4}="c4";%
   (15,20)*+{Z_5'}="c5";%
   (20,30)*+{Z_6}="c6";
    {\ar@{-}"c0";"c1"};
    {\ar@{-}"c0";"c2"};
    {\ar@{-}"c0";"c3"};
    {\ar@{-}"c1";"c4"};
    {\ar@{-}"c1";"c5"};
    {\ar@{-}"c2";"c6"};
    {\ar@{-}"c3";"c6"};
    {\ar@{-}"c4";"c6"};
    {\ar@{-}"c5";"c6"};    
   \end{xy}
\end{array}
\]
Define 
\[
   f(Z_0)=0,\ f(Z_1)=f(Z_2)=f(Z_3)=2\ \text{ and }f(Z_4)=f(Z_5)=f(Z_5')=3,\ f(Z_6)=4.
\]
Now one can check that $\cZ_i$ together with the map~$f$ satisfy the axioms (Z1)--(Z3) of cyclic flats given in \cite[Def.~4.1]{AlBy24} and therefore
give rise to \qM{}s $\cM_i=(E,\rho_i)$ such that $\cZ(\cM_i)=\cZ_i$; see \cite[Cor.~4.12]{AlBy24}.
The rank functions are given by 
\[
   \rho_i(V)=\min_{Z\in\cZ_i}\big(f(Z)+\dim((V+Z)/Z)\big) \ \text{ for all }V\in\cL(E);
\]
see also \eqref{e-RhoqM}, which in particular implies $\rho_i(Z)=f(Z)$ for $Z\in\cZ_i$ (\cite[Prop.~4.7]{AlBy24}).
The configurations of~$\cM_1$ and~$\cM_2$ are 
\pagebreak
\[
\begin{array}{l}
   \begin{xy}
   (-10,15)*+{\mbox{$\Config(\cM_1)=$}}="title";%
   (20,0)*+{(4,0)}="c0";%
   (10,10)*+{(2,1)}="c1";%
   (20,10)*+{(2,1)}="c2";%
   (30,10)*+{(2,1)}="c3";%
   (10,20)*+{(1,2)}="c4";%
   (20,20)*+{(1,2)}="c5";%
   (20,30)*+{(0,4)}="c6";
    {\ar@{-}"c0";"c1"};
    {\ar@{-}"c0";"c2"};
    {\ar@{-}"c0";"c3"};
    {\ar@{-}"c1";"c4"};
    {\ar@{-}"c2";"c5"};
    {\ar@{-}"c3";"c6"};
    {\ar@{-}"c4";"c6"};
    {\ar@{-}"c5";"c6"};    
   \end{xy}
\end{array}
\hspace*{2cm}
\begin{array}{l}
   \begin{xy}
   (-15,15)*+{\mbox{$\Config(\cM_2)=$}}="title";%
   (20,0)*+{(4,0)}="c0";%
   (10,10)*+{(2,1)}="c1";%
   (20,10)*+{(2,1)}="c2";%
   (30,10)*+{(2,1)}="c3";%
   (5,20)*+{(1,2)}="c4";%
   (15,20)*+{(1,2)}="c5";%
   (20,30)*+{(0,4)}="c6";
    {\ar@{-}"c0";"c1"};
    {\ar@{-}"c0";"c2"};
    {\ar@{-}"c0";"c3"};
    {\ar@{-}"c1";"c4"};
    {\ar@{-}"c1";"c5"};
    {\ar@{-}"c2";"c6"};
    {\ar@{-}"c3";"c6"};
    {\ar@{-}"c4";"c6"};
    {\ar@{-}"c5";"c6"};    
   \end{xy}
\end{array}
\]
The coarsest condensations of $\cZ_1$ and $\cZ_2$ are 
\[
   \cP_1=\big(\{Z_0\},\,\{Z_1,Z_2,Z_3\},\,\{Z_4,Z_5\},\,\{Z_6\}\big)\ \text{ and }
   \cP_1=\big(\{Z_0\},\,\{Z_1,Z_2,Z_3\},\,\{Z_4,Z_5'\},\,\{Z_6\}\big),
\]
respectively. They lead to the condensed configurations $\text{Config}^*_{\cP_i}(\cM_i)=( \Gamma_{\cP_i},\Lambda_{\cP_i})$, where 
\[
   \Gamma_{\cP_1}=\Gamma_{\cP_2}=\begin{pmatrix}1&1&1&1\\0&1&1&3\\0&0&1&2\\0&0&0&1\end{pmatrix}
   \ \text{ and }\
   \Lambda_{\cP_1}=\Lambda_{\cP_2}=\big((4,0),(2,1),(1,2),(0,4)\big).
\]
This example shows that the condensed configuration of a \qM{}~$\cM$ does not determine the configuration.
Nonetheless, the two \qM{}s have the same Whitney function $R_{\cM_1}=R_{\cM_2}=$
\[
  x^4 \!+\! 255x^3 \!+\! (3y + 10795)x^2 \!+\! (2y^2 \!+\! 149y \!+\! 97152)x \!+\! y^4 \!+\! 255y^3 \!+\! 10795y^2 \!+\! 97153y \!+\! 200638.
\]
Moreover, $c_{\cM_1,Z_1}\neq c_{\cM_2,Z_1}$ and $c_{\cM_1,Z_2}\neq c_{\cM_2,Z_2}$, while all other corresponding cloud and flock polynomials 
of~$\cM_1$ and~$\cM_2$ coincide.
Finally, the definition of condensations shows that condensed configurations do not behave well under duality: $\cM_1$ and~$\cM_2$ do not share a condensed configuration.
\end{exa}

In the rest of this section we will show that any condensed configuration determines the Whitney function.
To do so, we need some preparation.
We fix a \qM{} $\cM=(E,\rho)$.

\begin{lemma}\label{L-ZZPrime}
Let $Z,Z'\in\cZ(\cM)$ be such that $Z\leq Z'$ and set $r=\rho(Z')-\rho(Z)$ and $s=\dim(Z')-\dim(Z)$. Then 
\begin{alphalist}
\item $(\cM|Z')/Z\cong (\cM/Z)|(Z'/Z)$.
\item If $Z'$ is a cover of~$Z$ in $\cZ(\cM)$, that is, $Z<Z'$ and there is no cyclic flat $Z''$ such that $Z<Z''<Z'$, then 
        $(\cM|Z')/Z\cong \cU_{r,s}$.
\end{alphalist}
\end{lemma}

\begin{proof}
(a) Both \qM{}s have ground space $Z'/Z$ and in both cases each space $V/Z\leq Z'/Z$, where $Z\leq V\leq Z'$, has rank $\rho(V)-\rho(Z)$.
\\
(b) Let $V\in\cL(E)$ be such that $Z\leq V\leq Z'$. 
Then $\rho(V)-\rho(Z)\leq\rho(Z')-\rho(Z)=r$, and we want to show that 
$\rho(V)-\rho(Z)=\min\{r,\dim(V/Z)\}$.
Closure of~$Z'$ implies $\cl(V)\leq Z'$.
\\
(i) Suppose $\cl(V)=Z'$. Then $\rho(V)=\rho(Z')$ and thus $\rho(V)-\rho(Z)=r$. 
Furthermore, cyclicity of~$Z$ implies $Z\leq\cyc(V)$ \cite[Cor.~3.8]{GLJ24DSCyc} and thus $\dim Z-\rho(Z)\leq\dim V-\rho(V)$ by \cref{R-cyccl} and the monotonicity of the nullity function.
All of this tells us that $r\leq \dim(V/Z)$.
\\
(ii) Let $\cl(V)\lneq Z'$. Then $Z\leq\cyc(V)\leq\cyc(\cl(V))\lneq Z'$ together with the fact that
$\cyc(\cl(V))\in\cZ(\cM)$ (see \cref{R-cyccl}) implies $Z=\cyc(V)=\cyc(\cl(V))$.
As a consequence, $\dim V-\rho(V)=\dim Z-\rho(Z)$, and therefore $\rho(V)-\rho(Z)=\dim(V/Z)$.\end{proof}

Recall $c_{r,n}$ and $f_{r,n}$ from \cref{E-CloudFlockUniform}, the maps $d_x$ and $d_y$ 
from \cref{D-truncation}, and~$*$ from \cref{R-CloudFlock3}.
For any blocks $B,B'$ of~$\cP$ such that $B\leq B'$ we use the notation $[B,B']$ for the interval in the lattice~$\cP$.

\begin{defi}\label{D-Recursion}
Let $\cP$ be a condensation of~$\cZ(\cM)$. 
For each block~$B$ of~$\cP$ choose a fixed representative $Z_B\in B$.
Let $B,B'\in\cP$ be blocks such that $B\leq B'$.
Set $r=\rho(Z_{B'})-\rho(Z_B)$ and $s=\dim(Z_{B'})-\dim(Z_B)$.
We define the polynomials
\begin{align*}
   c_{\cP}(B,B')(x)&=\Gamma_{\cP}(B,B')c_{r,s}(x)-d_x(S(B,B')),\\[.7ex]
   f_{\cP}(B,B')(y)&=\Gamma_{\cP}(B,B')f_{r,s}(y)-d_y(S(B,B')),\\[.7ex]
     S(B,B')\ &=\sum_{D\in[B,B']\setminus\{B,B'\}}c_{\cP}(D,B')(x)*f_{\cP}(B,D)(y).
\end{align*}
This can be carried out recursively based on the length of the interval $[B,B']$.
 \end{defi}

Clearly, $S(B,B)=0$ and $c_{\cP}(B,B)=1=f_{\cP}(B,B)$.

\begin{prop}\label{P-Recursion}
Use the notation as in \cref{D-Recursion}. Choose blocks $B\leq B'$ of~$\cP$ and set $Z'=Z_{B'}$.
Then 
\begin{equation}\label{e-cfBBPrime}
      c_{\cP}(B,B')=\sum_{\substack{Z\in B\\ Z\leq Z'}} c_{\cM|Z',Z}\ \text{ and }\
      f_{\cP}(B,B')=\sum_{\substack{Z\in B\\ Z\leq Z'}} f_{\cM/Z,Z'/Z}.
\end{equation}
As a consequence, the right hand sides are independent of the choice of the representative $Z'=Z_{B'}$.
\end{prop}

\begin{proof}
We induct on the length of the interval $[B,B']$. 
\\
(1) For $B=B'$ we have $c_{\cP}(B,B)=1=c_{\cM|Z',Z'}$ and $f_{\cP}(B,B)=1=f_{\cM/Z',Z'/Z'}$.
\\
(2) Let $|[B,B']|=2$, thus $[B,B']=\{B,B'\}$. Then $S(B,B')=0$ and thus 
$c_{\cP}(B,B')=\Gamma_{\cP}(B,B')c_{r,s}$ and $f_{\cP}(B,B')=\Gamma_{\cP}(B,B')f_{r,s}$ by \cref{D-Recursion}.
Furthermore, in this case $Z'$ is a cover of every~$Z\in B$ satisfying $Z\leq Z'$.
With the aid of \cref{C-CloudFlockRestrContr}, \cref{L-ZZPrime}, and \cref{E-CloudFlockUniform}
we thus obtain  $c_{\cM|Z',Z}=c_{(\cM|Z')/Z,0}=c_{r,s}$ and $f_{\cM/Z,Z'/Z}=f_{(\cM/Z)|(Z'/Z),Z'/Z}=f_{r,s}$ for each such~$Z$.
Since there are $\Gamma_{\cP}(B,B')$ cyclic flats in~$B$ that are contained in~$Z'$, the right hand sides of \eqref{e-cfBBPrime} 
equal $\Gamma_{\cP}(B,B')c_{r,s}$ and $\Gamma_{\cP}(B,B')f_{r,s}$, as desired.
\\
(3) Let now $|[B,B']|>2$ and choose $D\in[B,B']\setminus\{B,B'\}$. Then we may use the induction hypothesis on the intervals 
$|[B,D]|$ and $|[D,B']|$.
In the following computation recall that the product~$*$ depends on the degrees of the factors.
For $Z''\in D$ such that $Z''\leq Z'$ we have $\deg(c_{\cM|Z',Z''})=\rho(Z')-\rho(Z'')$, and this does not depend on the choice of~$Z''$. 
Similarly, for all $Z\in B$ and $Z''\in D$ such that $Z\leq Z''$ we have $\deg(f_{\cM/Z,Z''/Z})=\dim(Z'')-\rho(Z'')-\dim(Z)+\rho(Z)$, and this does not 
depend on the choice of~$Z$ and~$Z''$.
For this reason, $\deg c_{\cP}(D,B)=\deg c_{\cM|Z',Z''}$ and similarly for $f_{\cP}(B,D)$, and this allows us to distribute factors in the computation below.
Using \cref{D-Recursion} and induction we compute
\begin{align}
     S(B,B')&=\sum_{D\in[B,B']\setminus\{B,B'\}}c_{\cP}(D,B')*f_{\cP}(B,D) \nonumber\\
         &=\sum_{D\in[B,B']\setminus\{B,B'\}}\Big[\sum_{\substack{Z''\in D\\ Z''\leq Z'}}c_{\cM|Z',Z''}\Big]
                          *\Big[\sum_{\substack{Z\in B\\Z\leq Z_D}}f_{\cM/Z,Z_D/Z}\Big] \nonumber\\
         &=\sum_{D\in[B,B']\setminus\{B,B'\}}\sum_{\substack{Z''\in D\\ Z''\leq Z'}}\Big(c_{\cM|Z',Z''}
                         *\Big[\sum_{\substack{Z\in B\\Z\leq Z_D}}f_{\cM/Z,Z_D/Z}\Big]\Big) \label{e-ThirdSum}\\
         &=\sum_{D\in[B,B']\setminus\{B,B'\}}\sum_{\substack{Z''\in D\\ Z''\leq Z'}}\sum_{\substack{Z\in B\\Z\leq Z''}}c_{\cM|Z',Z''}*f_{\cM/Z,Z''/Z},\nonumber
\end{align}
where the last step follows from the fact that the rightmost sum in \eqref{e-ThirdSum} is independent of the choice of $Z_D$ in~$D$ and thus we may choose 
$Z_D=Z''$. 
Thus
\[
   S(B,B')=\sum_{\substack{Z\in B\\Z\lneq Z'}}T(Z,Z'),\ \text{ where }\
              T(Z,Z')=\sum_{\substack{Z''\in\cZ(\cM)\\Z\lneq Z''\lneq Z'}}c_{\cM|Z',Z''}*f_{\cM/Z,Z''/Z}.
\]
Setting $r=\rho(Z_{B'})-\rho(Z_B)$ and $s=\dim(Z_{B'})-\dim(Z_B)$, we obtain from  \cref{D-Recursion}
\begin{equation}\label{e-LHS}
\left.\begin{split}
    c_{\cP}(B,B')&=\Gamma_{\cP}(B,B') c_{r,s}-d_x(S(B,B'))=
                 \sum_{\substack{Z\in B\\Z\lneq Z'}}\!\!\big[c_{r,s}-d_x(T(Z,Z'))\big],\\
     f_{\cP}(B,B')&=\Gamma_{\cP}(B,B') f_{r,s}-d_y(S(B,B'))=
                  \sum_{\substack{Z\in B\\Z\lneq Z'}}\!\!\big[f_{r,s}-d_y(T(Z,Z'))\big].    
\end{split}\quad\right\}
\end{equation}
We now evaluate the right hand sides of \eqref{e-cfBBPrime}.
Thanks to \cref{C-CloudFlockRestrContr} and \cref{L-ZZPrime} we have $c_{\cM|Z',Z}=c_{(\cM|Z')/Z,0}$ and $f_{\cM/Z,Z'/Z}=f_{(\cM|Z')/Z,Z'/Z}$.
By \cref{L-RestrContr} the \qM{} $(\cM|Z')/Z$ is full and we may apply \cref{P-TruncWhitney}.
We thus obtain for any $Z\in B$ such that $Z\leq Z'$
\begin{equation}\label{e-RHS}
\left.\begin{split}
   c_{\cM|Z',Z}&=c_{r,s}-d_x\big(\!\!\sum_{Z''/Z\in\tilde{\cZ}}\!\!c_{(\cM|Z')/Z,Z''/Z}*f_{(\cM|Z')/Z,Z''/Z}\big)
      =c_{r,s}-d_x(T'(Z,Z')),  \\
   f_{\cM/Z,Z'/Z}&=f_{r,s}-d_y\big(\!\!\sum_{Z''/Z\in\tilde{\cZ}}\!\!\!c_{(\cM|Z')/Z,Z''/Z}*f_{(\cM|Z')/Z,Z''/Z}\big)
      =f_{r,s}-d_y(T'(Z,Z')),
\end{split}\quad\right\}
\end{equation}
where $\tilde{\cZ}=\cZ((\cM|Z')/Z)\setminus\{0,Z'/Z\}$ and 
\[
   T'(Z,Z')=\sum_{\substack{Z''\in\cZ(\cM)\\Z\lneq Z''\lneq Z'}}\!\!\!c_{(\cM|Z')/Z,Z''/Z}*f_{(\cM|Z')/Z,Z''/Z}.
\]
Note that $((\cM|Z')/Z)/(Z''/Z)\cong(\cM|Z')/Z''$ and therefore, thanks to \cref{C-CloudFlockRestrContr},
\begin{align*}
   c_{(\cM|Z')/Z,Z''/Z}&=c_{((\cM|Z')/Z)/(Z''/Z),0}=c_{(\cM|Z')/Z'',0}=c_{\cM|Z',Z''},\\
   f_{(\cM|Z')/Z,Z''/Z}&=f_{(\cM/Z)|(Z'/Z),Z''/Z}=f_{\cM/Z,Z''/Z}.
\end{align*}
All of this shows  that $T(Z,Z')=T'(Z,Z')$ and \eqref{e-LHS} and \eqref{e-RHS} establish the desired identities.
\end{proof}

\begin{rem}\label{R-FlockCondens}
If $\cl(0)=0$ (that is,~$\cM$ is loopless), then $0_{\cP}=\{0\}$, and \cref{P-Recursion} implies that $f_{\cP}(0_{\cP},B)=f_{\cM,Z}$ 
for any $Z\in B$. 
In other words, all cyclic flats in a fixed block of a condensation have the same flock polynomial.
The same statement is not true for the cloud polynomials. Indeed, $c_{\cM_1,Z_1}\neq c_{\cM_1,Z_3}$ in \cref{E-CConfig}.
\end{rem}

\begin{theo}\label{T-CondensWhitney}
Let~$\cP$ be a condensation of~$\cZ(\cM)$. 
\begin{alphalist}
\item If~$\cM$ is full, then
        \[
            R_\cM=\sum_{B\in\cP}c_{\cP}(B,1_{\cP})*f_{\cP}(0_{\cP},B).
        \]
\item $R_{\cM}$ is fully determined by the condensed configuration $\CConfig(\cM)$.
\end{alphalist}
\end{theo}

\begin{proof}
(a) Let~$\cM$ be full, thus $1_{\cP}=\{E\}$ and $0_{\cP}=\{0\}$.
In this case we have $Z'=E$ in the expression for $c_{\cP}(B,1_{\cP})$ in \cref{P-Recursion}, while 
$f_{\cP}(0_{\cP},B)=f_{\cM,Z}$ for any $Z\in B$. Thus
\[
   \sum_{B\in\cP}c_{\cP}(B,1_{\cP})*f_{\cP}(0_{\cP},B)=\sum_{B\in\cP}\big[\sum_{Z\in B}c_{\cM,Z}\big]*\big[f_{\cM,Z}\big]
     =\sum_{Z\in\cZ(\cM)}c_{\cM,Z}*f_{\cM,Z}=R_\cM
\]
where the last step follows from \cref{R-CloudFlock3}.
\\
(b) Suppose first that~$\cM$ is full. By \cref{D-Recursion} the polynomials $c_{\cP}(B,1_{\cP})$ and $f_{\cP}(0_{\cP},B)$ can be computed from 
$\CConfig(\cM)$, and hence the statement follows.
Suppose now that~$\cM$ is not full, thus $1_{\cP}=\{\cyc(E)\}$ and $0_{\cP}=\{\cl(0)\}$.
From \cite[Thm.~7.9]{GLJ24DSCyc} it is known that~$\cM$ has a decomposition of the form
$\cM=\cM_1\oplus\cU_{0,n_2}(E_2)\oplus\cU_{n_3,n_3}(E_3)$, where $\cM_1=(E_1,\rho_1)$ is full,  
$E_2=\cl(0)$, and $E_3\in\cL(E)$ is such that $E=E_1\oplus E_2\oplus E_3$.
Moreover, $\cyc(E)=E_1\oplus E_2$.
By Propositions~\ref{P-FreeSummand} and~\ref{P-TrivialSummand} we have
$\cZ(\cM)=\{Z_1\oplus E_2\oplus 0\mid Z_1\in\cZ(\cM_1)\}$.
Fix a block~$B$ of~$\cP$. Then 
\[
    c_{\cP}(B,1_{\cP})=\sum_{Z_1\oplus E_2\in B}c_{\cM|\cyc(E),Z}=\sum_{Z_1\oplus E_2\in B}c_{\cM_1\oplus\cU_{0,n_2}(E_2),Z_1\oplus E_2}
    =\sum_{Z_1\oplus E_2\in B}c_{\cM_1,Z_1},
\]
where the last step follows from \cref{P-TrivialSummand}(c).
Furthermore, using that $\cl(0)=E_2$ we have for any $Z=Z_1\oplus E_2\in B$
\[
   f_{\cP}(0_{\cP},B)=f_{\cM/E_2,Z/E_2}=f_{\cM_1\oplus\cU_{n_3,n_3}(E_3),Z_1}=f_{\cM_1,Z_1},
\]
where the second step follows from \cite[Cor.~49]{CeJu24} and the last one from \cref{P-FreeSummand}(b).
Now we arrive at 
\[
    \sum_{B\in\cP} c_{\cP}(B,1_{\cP})*f_{\cP}(0_{\cP},B)=\sum_{B\in\cP}\sum_{Z_1\oplus E_2\in B}\!\!\!\!c_{\cM_1,Z_1}*f_{\cM_1,Z_1}
    =\!\!\sum_{Z_1\in\cZ(\cM_1)}\!\!\!c_{\cM_1,Z_1}*f_{\cM_1,Z_1}=R_{\cM_1}.
 \]
 This shows that $R_{\cM_1}$ can be computed from $\CConfig(\cM)$. But then the same is true for $R_{\cM}$ thanks to 
 Propositions~\ref{P-FreeSummand} and \ref{P-TrivialSummand}.
\end{proof}

\cref{E-DiffCycFlats} in the next section shows that the converse of \cref{T-CondensWhitney}{(b) is not true.

\section{Information Retrievable from the Whitney Function}\label{S-Converse}

In this section we briefly investigate which information about the cyclic flats and cloud and flock polynomials can be retrieved from the Whitney function. 
First of all, the converse of \cref{T-ConfigWhitney} is not true, that is, the Whitney function does not determine the configuration. 
Even worse, it does not even determine the number of cyclic flats.
This is not surprising because the same is true for matroids. 
However, as we will see below, our counterexample can be extended in order to illustrate a
substantial difference between \qM{}s and matroids when it comes to the direct sum. 

\begin{exa}\label{E-DiffCycFlats}
Consider the $\F=\F_{2^6}$ with primitive element~$\omega$ with minimal polynomial $x^6+x^4+x^3+x+1\in\F_2[x]$.
Let
\[
   G_1=\begin{pmatrix}1&0&0&\omega^{44}&\omega^{34}&0\\ 0&1&0&\omega^{37}&\omega^{9}&1\\0&0&1&\omega^7&\omega^5&\omega\end{pmatrix},
   \quad
   G_2=\begin{pmatrix}1&0&0&\omega^{31}&\omega^{22}&0\\ 0&1&0&\omega^{56}&\omega^{30}&1\\0&0&1&\omega^7&\omega^5&\omega\end{pmatrix},   
\]
and $\cM_i=\cM_{G_i}=(\F_2^6,\rho_i)$ for $i=1,2$. It turns out that
\[
  R_{\cM_1}=R_{\cM_2}=x^3 + (y + 63)x^2 + (y^2 + 39y + 650)x + y^3 + 63y^2+ 650y + 1356.
\]
On the other hand, $|\cZ(\cM_1)|=16$ and $|\cZ(\cM_2)|=13$.
The configurations are 
{\footnotesize
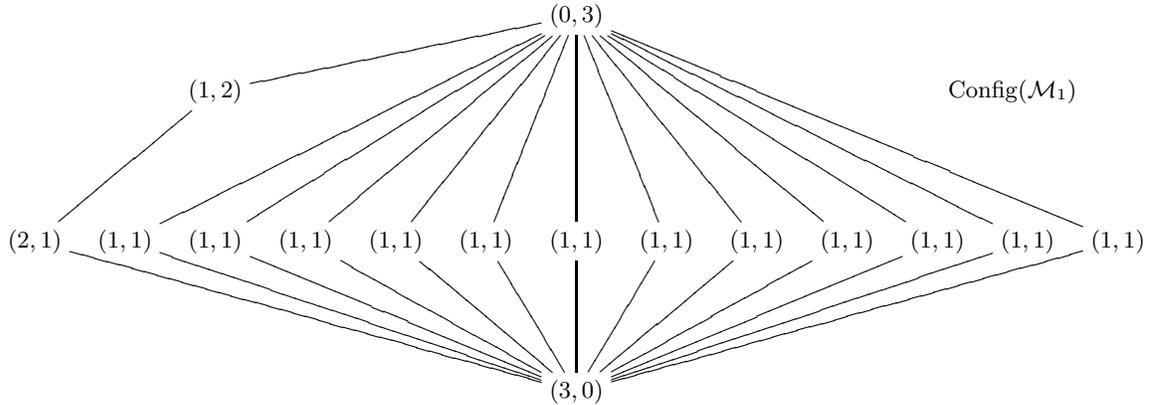
\[
\begin{array}{l}
   \begin{xy}
   (130,30)*+{\Config(\cM_1)}="title";%
   (72,-10)*+{(3,0)}="c0";%
   (0,10)*+{(2,1)}="c1";%
   (12,10)*+{(1,1)}="c2";%
   (24,10)*+{(1,1)}="c3";%
   (36,10)*+{(1,1)}="c4";%
   (48,10)*+{(1,1)}="c5";%
   (60,10)*+{(1,1)}="c6";%
   (72,10)*+{(1,1)}="c7";%
   (84,10)*+{(1,1)}="c8";%
   (96,10)*+{(1,1)}="c9";%
   (108,10)*+{(1,1)}="c10";%
   (120,10)*+{(1,1)}="c11";%
   (132,10)*+{(1,1)}="c12";%
   (144,10)*+{(1,1)}="c13";%
   (24,30)*+{(1,2)}="c14";%
   (72,40)*+{(0,3)}="c15";%
    {\ar@{-}"c0";"c1"};
    {\ar@{-}"c0";"c2"};
    {\ar@{-}"c0";"c3"};
    {\ar@{-}"c0";"c4"};
    {\ar@{-}"c0";"c5"};
    {\ar@{-}"c0";"c6"};
    {\ar@{-}"c0";"c7"};
    {\ar@{-}"c0";"c8"};
    {\ar@{-}"c0";"c9"};
    {\ar@{-}"c0";"c10"};
    {\ar@{-}"c0";"c11"};
    {\ar@{-}"c0";"c12"};
    {\ar@{-}"c0";"c13"};
    {\ar@{-}"c1";"c14"};
     {\ar@{-}"c14";"c15"};
    {\ar@{-}"c2";"c15"};
    {\ar@{-}"c3";"c15"};
    {\ar@{-}"c4";"c15"};
    {\ar@{-}"c5";"c15"};
    {\ar@{-}"c6";"c15"};
    {\ar@{-}"c7";"c15"};
    {\ar@{-}"c8";"c15"};
    {\ar@{-}"c9";"c15"};
    {\ar@{-}"c10";"c15"};
    {\ar@{-}"c11";"c15"};
    {\ar@{-}"c12";"c15"};
    {\ar@{-}"c13";"c15"};
   \end{xy}
\end{array}
\]
}
and
{\footnotesize
\[
\begin{array}{l}
   
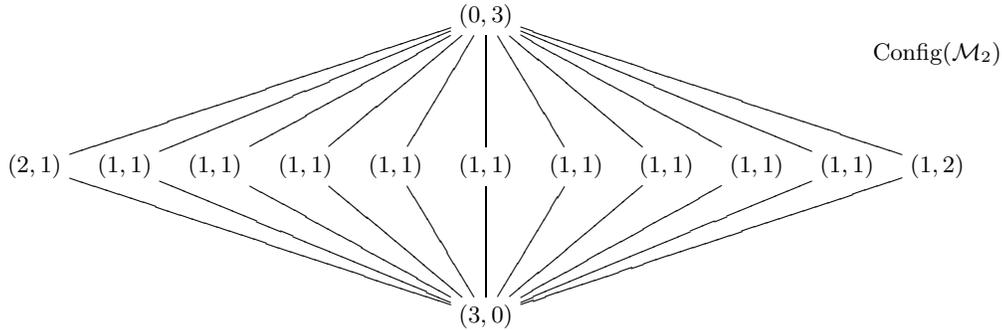
\begin{xy}
   (120,25)*+{\Config(\cM_2)}="title";%
   (60,-10)*+{(3,0)}="c0";%
   (0,10)*+{(2,1)}="c1";%
   (12,10)*+{(1,1)}="c2";%
   (24,10)*+{(1,1)}="c3";%
   (36,10)*+{(1,1)}="c4";%
   (48,10)*+{(1,1)}="c5";%
   (60,10)*+{(1,1)}="c6";%
   (72,10)*+{(1,1)}="c7";%
   (84,10)*+{(1,1)}="c8";%
   (96,10)*+{(1,1)}="c9";%
   (108,10)*+{(1,1)}="c10";%
   (120,10)*+{(1,2)}="c11";%
   (60,30)*+{(0,3)}="c15";%
    {\ar@{-}"c0";"c1"};
    {\ar@{-}"c0";"c2"};
    {\ar@{-}"c0";"c3"};
    {\ar@{-}"c0";"c4"};
    {\ar@{-}"c0";"c5"};
    {\ar@{-}"c0";"c6"};
    {\ar@{-}"c0";"c7"};
    {\ar@{-}"c0";"c8"};
    {\ar@{-}"c0";"c9"};
    {\ar@{-}"c0";"c10"};
    {\ar@{-}"c0";"c11"};
    {\ar@{-}"c1";"c15"};
    {\ar@{-}"c2";"c15"};
    {\ar@{-}"c3";"c15"};
    {\ar@{-}"c4";"c15"};
    {\ar@{-}"c5";"c15"};
    {\ar@{-}"c6";"c15"};
    {\ar@{-}"c7";"c15"};
    {\ar@{-}"c8";"c15"};
    {\ar@{-}"c9";"c15"};
    {\ar@{-}"c10";"c15"};
    {\ar@{-}"c11";"c15"};
   \end{xy}
\end{array}
\]
}
\end{exa}

We now turn to the direct sum.
Recall that for classical matroids it is known that the Whitney function of a direct sum $M_1\oplus M_2$ is the product of the Whitney 
functions of the two summands.
One easily verifies that this is not the case for \qM{}s, and it is also not true for the Tutte polynomials as defined in \cite{ByFu25}.
In fact, as following example shows the situation is even worse: for \qM{}s~$\cM_1$ and $\cM_2$, 
the Whitney function $R_{\cM_1\oplus\cM_2}$ is not a function of $R_{\cM_1}$ and $R_{\cM_2}$.

\begin{exa}\label{E-WhitDS}
Consider  \cref{E-DiffCycFlats}. The \qM{}s $\cM_1$ and $\cM_2$  have the same Whitney function, yet taking the direct sum of~$\cM_i$ with the uniform \qM{}
$\cU_{1,2}(\F_2^2)$ leads to the Whitney functions
\begin{align*}
   R_{\cU_{1,2}(\F_2^2)\oplus\cM_1}&=x^4 + (2y + 255)x^3 + (2y^2 + 159y + 10793)x^2 + (2y^3 + 159y^2 + 3469y+ 96996)x \\
          &\quad + y^4 + 255y^3 + 10793y^2 + 96996y +    197316,\\
   R_{\cU_{1,2}(\F_2^2)\oplus\cM_2}&=x^4 + (2y + 255)x^3 + (2y^2 + 159y + 10793)x^2 + (2y^3 + 159y^2 + 3475y+ 96996)x \\
         &\quad + y^4 + 255y^3 + 10793y^2 + 96996y +    197310.
\end{align*}
This shows that there is no $q$-analogue of the well-known matroid identity $R_{M_1\oplus M_2}=R_{M_1}R_{M_2}$.
In \cite{ByFu25} the authors derive a notion of Tutte polynomial for \qM{}s that admit a Tutte partition and establish a bijection between the Whitney function and the Tutte polynomial~\cite[Prop.~67 and Cor.~68]{ByFu25}. 
It can be expressed in terms of a rather technical ($\Z$-linear) substitution for the monomials $x^iy^j$. 
Thanks to this bijection, our example shows that there is also no $q$-analogue of $R_{M_1\oplus M_2}=R_{M_1}R_{M_2}$ in terms of Tutte polynomials.
Finally, the direct sums $\cU_{1,2}(\F_2^2)\oplus\cM_1$ and $\cU_{1,2}(\F_2^2)\oplus\cM_2$ also have different characteristic polynomials.
This follows immediately from the substitution rule given in \cref{R-CharPoly}. 
Hence the Whitney functions of two \qM{}s do not even determine the characteristic polynomial of their direct sum.

\end{exa}

Despite these negative results, there is some partial information about the cyclic flats that can be extracted from the Whitney function.
We need the notion of extremal terms in a polynomial; see also \cite[Rem.~4.11]{PlBae14} for matroids.

\begin{defi}\label{D-fMat}
Let $f=\sum_{i=0}^a\sum_{j=0}^bf_{i,j}x^iy^j\in\Z[x,y]$ with $\deg_x f=a$ and $\deg_y f=b$. We define 
\[
    \M_f=\begin{pmatrix}f_{0,0}&f_{0,1}&\ldots&f_{0,b}\\ \vdots&\vdots& &\vdots\\ f_{a,0}&f_{a,1}&\ldots&f_{a,b}\end{pmatrix}\in\Z^{(a+1)\times(b+1)}.
 \]
A position $(\alpha,\beta)\in\{0,\ldots,a\}\times\{0,\ldots,b\}$ is called \textbf{extremal} in~$M_f$ if 
\[
    f_{\alpha,\beta}\neq0,\quad f_{\alpha,j}=0\text{ for all }j>\beta,\quad f_{i,\beta}=0\text{ for all }i>\alpha.
\]
In other words, $ f_{\alpha,\beta}$ is the rightmost nonzero entry in its row and the lowest nonzero entry in its column.
In this case we call $f_{\alpha,\beta}x^\alpha y^\beta$ an \textbf{extremal term} of~$f$.
\end{defi}

Before we turn to the extremal terms of the Whitney function, we show that its coefficient matrix is  top-left aligned.

\begin{prop}\label{P-WhitMat}
Let $\cM=(E,\rho)$ be a \qM{} of rank~$\hat{\rho}$ on the $n$-dimensional ground space~$E$, 
and let $R_\cM=\sum_{i=0}^{\hat{\rho}}\sum_{j=0}^{n-\hat{\rho}}\nu_{i,j}x^iy^j$. 
We define the \textbf{Whitney matrix} of~$\cM$ as
\[
      \W_\cM:=\M_{R_{\cM}}\in\Z^{(\hat{\rho}+1)\times(n-\hat{\rho}+1)}.
\]
Let $(\alpha,\beta)$ be such that $\nu_{\alpha,\beta}\neq0$. Then
\begin{alphalist}
\item $\nu_{\alpha,j}\neq0$ for all $j=0,\ldots,\beta$.
\item $\nu_{i,\beta}\neq0$ for all $i=0,\ldots,\alpha$.
\end{alphalist}
Hence the nonzero entries of $\W_{\cM}$ are top-left aligned.
\end{prop}

\begin{proof}
a) If $\beta=0$, there is nothing to show. 
Thus let~$\beta>0$. 
By assumption there exists $V_0\in\cL(E)$ such that $\rho(V_0)=\hat{\rho}-\alpha$ and $\dim V_0=\hat{\rho}-\alpha+\beta$.
Then $V_0$ is dependent and there exists a basis $A\lneq V_0$. Hence $\dim A=\rho(A)=\rho(V_0)=\hat{\rho}-\alpha$.
This shows that $\nu_{\alpha,0}\neq0$. 
Moreover, for all $j\in\{0,\ldots,\beta\}$ there exists a subspace $B:=A\oplus\subspace{v_1,\ldots,v_{j}}\leq V_0$, and it satisfies 
$\rho(B)=\hat{\rho}-\alpha$ and $\dim B=\hat{\rho}-\alpha+j$. Thus $\nu_{\alpha,j}\neq0$.
\\
b) Thanks to \cref{T-WhitDual} we have $\W_{\cM^*}=(\W_{\cM})\T$ and the result follows from~(a).
\end{proof}

Now we arrive at the following $q$-analogue of \cite[Rem.~4.11]{PlBae14}.

\begin{theo}\label{T-Extremal}
Let $\cM$ be a \qM{} of rank~$\hat{\rho}$ on an $n$-dimensional ground space and consider its Whitney matrix
$\W_\cM\in\Z^{(\hat{\rho}+1)\times(n-\hat{\rho}+1)}$.
Suppose  $\nu_{\alpha,\beta}x^\alpha y^\beta$ is an extremal term of $R_\cM$.
Then there exist exactly $\nu_{\alpha,\beta}$ cyclic flats of corank~$\alpha$ and nullity~$\beta$.
\end{theo}

\begin{proof}
Define
\begin{align*}
   \cA&=\{V\in\cL(E)\mid \rho(V)=\hat{\rho}-\alpha,\,\dim V=\rho(V)+\beta\},\\
   \cB&=\{V\in\cL(E)\mid \rho(V)<\hat{\rho}-\alpha,\,\dim V=\rho(V)+\beta\},\\
   \cC&=\{V\in\cL(E)\mid \rho(V)=\hat{\rho}-\alpha,\,\dim V>\rho(V)+\beta\}.
\end{align*}
The extremality of $(\alpha,\beta)$ implies $|\cA|=\nu_{\alpha,\beta}>0$ and $\cB=\cC=\emptyset$.
Let $Z\in\cA$. We have to show that~$Z$ is a cyclic flat.
Since $\cC=\emptyset$, there are no spaces properly containing~$Z$ with the same rank, and thus~$Z$ is a flat.
Suppose now that~$Z$ is not cyclic. Then there exists $W\in\Hyp(Z)$ such that $\rho(W)=\rho(Z)-1$.
Hence $\dim W-\rho(W)=\dim Z-\rho(Z)=\beta$ and $\rho(W)=\hat{\rho}-\alpha-1$.
Thus $W\in\cB$, which is a contradiction.
Finally, any cyclic flats of corank~$\alpha$ and nullity~$\beta$ has to be an element of~$\cA$, showing that there are exactly $\nu_{\alpha,\beta}$ 
such cyclic flats.
\end{proof}

The above result does not account for all cyclic flats.

\begin{exa}\label{E-DiffCycFlatsA}
Consider \cref{E-DiffCycFlats}.
We have
\[
    \W_{\cM_1}=\W_{\cM_2}=\begin{pmatrix}1356&650&63&1\\650&39&1&0\\63&1&0&0\\1&0&0&0\end{pmatrix}
\]
There are 4 extremal positions and they account, in either of the two \qM{}s, for the unique cyclic flats with corank-nullity data 
$(0,3),\,(1,2),\,(2,1)$, and $(3,0)$.

\end{exa}

\begin{exa}\label{E-CloudFlock2A}
In \cref{E-CloudFlock2} the configuration of $\cM$ is given by 
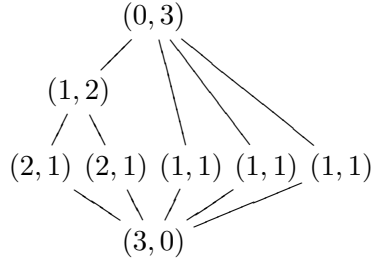
\[
\begin{array}{l}
   \begin{xy}
   (25,0)*+{(3,0)}="c0";%
   (10,10)*+{(2,1)}="c1";%
   (20,10)*+{(2,1)}="c2";%
   (30,10)*+{(1,1)}="c3";%
   (40,10)*+{(1,1)}="c4";%
   (50,10)*+{(1,1)}="c5";%
   (15,20)*+{(1,2)}="c6";
   (25,30)*+{(0,3)}="c7";%
    {\ar@{-}"c0";"c1"};
    {\ar@{-}"c0";"c2"};
    {\ar@{-}"c0";"c3"};
    {\ar@{-}"c0";"c4"};
    {\ar@{-}"c0";"c5"};
    {\ar@{-}"c1";"c6"};
    {\ar@{-}"c2";"c6"};
    {\ar@{-}"c3";"c7"};
    {\ar@{-}"c4";"c7"};
    {\ar@{-}"c5";"c7"};    
    {\ar@{-}"c6";"c7"};
   \end{xy}
\end{array}
\]
and the Whitney matrix is 
\[
   \W_\cM=\begin{pmatrix}1353 & 650 & 63 & 1  \\ 649 & 42 & 1 & 0  \\ 63 & 2 & 0 & 0  \\ 1 & 0 & 0 & 0 \end{pmatrix}.
\]
We have 4 extremal positions, and they are in compliance with the 5 cyclic flats with corank-nullity data $(0,3),(1,2),(2,1),(2,1),(3,0)$.
\end{exa}

\begin{exa}\label{E-UnifWhitneyMat}
Let $\cM=\cU_{k,n}(E)$, where $1\leq k\leq n-1$. Then $\cZ(\cM)=\{0,E\}$ and the Whitney function is 
$R_{\cM}=\sum_{j=0}^k\Gaussian{n}{k-j}x^j+\sum_{j=1}^{n-k}\Gaussian{n}{j+k}y^j$.
Thus the Whitney matrix is 
\[
  \W_\cM=\begin{pmatrix}\Gaussian{n}{k}&\Gaussian{n}{k+1}&\cdots&\Gaussian{n}{n}\\ 
                                         \Gaussian{n}{k-1}&0&\cdots&0\\ \vdots& \vdots& &\vdots\\   \Gaussian{n}{0}&0&\cdots&0 \end{pmatrix}.
\]
We have two extremal positions, implying the cyclic flats $0$ and~$E$.
In this case, this accounts for all cyclic flats.
\end{exa}

\begin{exa}\label{E-CConfig2}
One easily verifies that in \cref{E-CConfig} the extremal terms of~$\W_{\cM_i}$ account for all cyclic flats of~$\cM_i$. 
\end{exa}

As \cref{E-DiffCycFlats} suggests, no further information about the cyclic flats can be retrieved in general from the Whitney function.

\section{Further Question}
Our investigation gives rise to the following questions, which we have to leave to future research.
\begin{alphalist}
\item The configuration fully determines the Whitney function. The proof, however, relies on induction and is not explicit.
         Can one give an explicit expression of the Whitney function in terms of the configuration?
         The same question applies to condensed configurations.
\item The Whitney function guarantees, via its extremal terms, the existence of certain cyclic flats. 
         Can these cyclic flats be distinguished from the remaining cyclic flats other than via their `extremality'?
\end{alphalist}

\appendix
\section*{Appendix}\label{Appendix}
\setcounter{section}{1}
\setcounter{theo}{0}

\textit{Proof of \cref{L-DirCompl}:}
Without loss of generality let $F=\F_q^f$ and $Z=\rs(0\mid I_z)$.
\\
(a) Let $W=\rs(W_1\mid W_2)\in\cL(F)$, where $(W_1\mid W_2)$ is a matrix with full row rank and $W_1$ (resp.~$W_2$) consists of $f-z$ (resp.~$z$) columns. 
Then 
\begin{align*}
   Z\oplus W=F&\Longleftrightarrow \dim W=f-z \text{ and }W_1\in\GL_{f-z}(\F_q)\\
      &\Longleftrightarrow W=\rs(I_{f-z}\mid B)\text{ for some }B\in\F_q^{(f-z)\times z}.
\end{align*}
Clearly, different matrices $B$ lead to different subspaces~$W$, and thus we have $q^{(f-z)z}$ direct complements.
\\
(b) Without loss of generality we may assume that $V=\rs(0\mid A)$, where $A=(I_v\mid A_2)$. Clearly, $V\cap W=\{0\}$ for all $W\in\cW$. Moreover, 
if $W=\rs(I_{f-z}\mid B)=\rs(I_{f-z}\mid B_1\mid B_2)$, then
\[
   W\oplus V=\rs\begin{pmatrix}I_{f-z}&B_1&B_2\\ 0&I_v&A_2\end{pmatrix}=\rs\begin{pmatrix}I_{f-z}&0&B_2-B_1A_2\\0&I_v&A_2\end{pmatrix}.
\]
The rightmost matrix is the unique representation of $W\oplus V$ in RREF. 
This implies that the set $\{V\oplus W\mid W\in\cW\}$ is in 
bijection to the subset $\{\rs(I_{f-z}\mid 0\mid B_2)\mid B_2\in\F_q^{(f-z)\times(z-v)}\}$, and the latter has cardinality $q^{(f-z)(z-v)}$.
Now the rest is clear.
\hfill$\square$

\medskip

\textit{Proof of \cref{L-RestrContr}:}
(a) ``$\supseteq$'' is clear.
For ``$\subseteq$'' let $F\in\cF(\cM|\hat{F})$. If $x\in\hat{F}\setminus F$, then clearly $\rho(F)<\rho(F+\subspace{x})$. 
If $x\in E\setminus\hat{F}$, then $\rho(\hat{F})<\rho(\hat{F}+\subspace{x})$, and thus by 
\cite[Lem.~3.2]{BCIJS23} $\rho(F)<\rho(F+\subspace{x})$. Hence $F\in\cF(\cM)$.
\\
(b) can be found in \cite[Prop.~2.10(2)]{Ja23} .
\\
(c) is obvious because cyclicity of~$O$ is based on hyperplanes (in the sense of Linear Algebra).
\\
(d) 
Set $\hat{\cM}=\cM/\hat{O}$. For ``$\supseteq$'' let $O\in\cO(\cM)$ and $\hat{O}\leq O$. 
Let $D/\hat{O}\in\Hyp(O/\hat{O})$. Then $D\in\Hyp(O)$ and thus 
$\rho_{\hat{\cM}}(D/\hat{O})=\rho(D)-\rho(\hat{O})=\rho(O)-\rho(\hat{O})=\rho_{\hat{\cM}}(O/\hat{O})$.
For ``$\subseteq$'' assume $O/\hat{O}$ is in $\cO(\hat{\cM})$. 
To show that $O$ is cyclic in~$\cM$, let $D\in\Hyp(O)$. 
If $\hat{O}\leq D$, then $D/\hat{O}\in\Hyp(O/\hat{O})$ and $\rho(D)-\rho(\hat{O})=\rho(O)-\rho(\hat{O})$, thus $\rho(D)=\rho(O)$.
If $\hat{O}\not\leq D$, then there exists $x\in\hat{O}$ such that $D+\subspace{x}=O$.
Since $\hat{O}$ is cyclic, it is contained in $\cyc(O)$ (see \cite[Thm.~3.6]{GLJ24DSCyc}), and thus the very definition of the cyclic core 
(\cite[Def.~3.3]{GLJ24DSCyc}) implies $\rho(D)=\rho(O)$. All of this shows that $O\in\cO(\cM)$.
\\
(e) First of all, $0$ is in $\cZ(\cM')$ thanks to~(b) and~(d).
Moreover, by~(b) the flats of~$\cM'$ are the spaces $F/\hat{Z}$, where~$F$ is a flat in~$\cM$ containing $\hat{Z}$.
Next, denoting the rank function of~$\cM'$ by~$\rho'$ and using \cref{P-CloudFlockProp} twice, we obtain for any $F\in\cF(\cM)$ containing $\hat{Z}$
\begin{align*}
  \qquad\quad F/\hat{Z}\in\cyc_{\cM'}^{-1}(0)&\Longleftrightarrow\dim(F/\hat{Z})-\rho'(F/\hat{Z})=0\Longleftrightarrow \dim F-\rho(F)=\dim\hat{Z}-\rho(\hat{Z})\\
  &\Longleftrightarrow F\in\cyc_{\cM}^{-1}(\hat{Z}).\hspace*{8.4cm}\square
\end{align*}

\medskip

\textit{Proof of \cref{L-ExtendSpaces}:}
Throughout, let $E_i=\F_q^{n_i}$ and $E=\F_q^n$, where $n=n_1+n_2$.
Then $\pi_1:\F^n\longrightarrow \F^{n_1}$  and $\pi_2:\F^n\longrightarrow \F^{n_2}$ are the projections onto the first~$n_1$ and
last~$n_2$ coordinates, respectively.
Furthermore, let $V_i=\rs(M_i)$ for matrices $M_i\in\F^{k_i\times n_i}$ in RREF.
\\
(a) Clearly, any $V\in\cV$ has dimension at least~$k_1$ and since it is contained in $V_1\oplus E_2$, its dimension is at most $k_1+n_2$.
This proves the second line.
Choose now $V\leq E$  such that $\dim V=j\in\{k_1,\ldots,k_1+n_2\}$. 
Then $V\in\cV_1$ iff its matrix representation in RREF is of the form
\[
   \begin{pmatrix}M_1&T_{21}\\ 0&T_{22}\end{pmatrix}
\]
for some $T_{22}\in\F^{(j-k_1)\times n_2}$ of rank $j-k_1$ in RREF and $T_{21}\in\F^{k_1\times n_2}$ with zeros in the pivot columns of $T_{22}$.
Hence the set~$\cV_1$ is in bijection with the set of all choices of $T_{21}$ and $T_{22}$ subject to these conditions.
There are $\Gaussian{n_2}{j-k_1}_q$ choices for $T_{22}$ and $q^{k_1(n_2-j+k_1)}$ choices for~$T_{21}$.
This proves the statement.
\\
(b) Clearly, $\dim\pi_1(V)\in\{0,\ldots,n_1\}$ for any space $V\in\cV'$. Fix $j\in\{0,\ldots,n_1\}$ and let $V\leq E$. 
Then $V\in\cV'$ if and only if 
\[
  V=\rs\begin{pmatrix}T_1&\hat{T}\\0&M_2\end{pmatrix}
\]
for some $T_1\in\F^{j\times n_1}$ in RREF of rank~$j$ and $\hat{T}\in\F^{j\times n_2}$ with zeros in the pivot columns of~$M_2$.
The set~$\cV'$ is in bijection with the set of all choices of~$T_1$ and $\hat{T}$ subject to these conditions.
There are $\Gaussian{n_1}{j}_q$ choices for~$T_1$ and $q^{j(n_2-k_2)}$ choices for $\hat{T}$.
\hfill$\square$

\medskip

\textit{Proof of \cref{L-FlatsFree}:}
(a) The first identity of \eqref{e-DimRhoFree} is basic Linear Algebra and the second one is given in 
\cite[Proof of Prop.~3.9]{GLJ22Rep}.
We now turn to the identity for flats.
\\
``$\subseteq$'' Let $F\in\cF(\cM)$ and $x_1\in E_1\setminus(F\cap E_1)$. Then $x_1\in E\setminus F$. Hence $\rho(F+\subspace{x_1})>\rho(F)$.
Since $\pi_2(F+\subspace{x_1})=\pi_2(F)$, \eqref{e-DimRhoFree} implies
$\rho_1((F\cap E_1)+\subspace{x_1})=\rho_1((F+\subspace{x_1})\cap E_1)>\rho_1(F\cap E_1)$.
Thus $F\cap E_1\in\cF(\cM_1)$.
\\
``$\supseteq$''
Let $F\in\cL(E)$ be such that $F\cap E_1\in\cF(\cM_1)$ and  let $x\in E\setminus F$.
Suppose first that $x\in E_1$. Then $x\in E_1\setminus(F\cap E_1)$ and thus $\rho_1((F+\subspace{x})\cap E_1)=\rho_1((F\cap E_1)+\subspace{x})>\rho_1(F\cap E_1)$.
Thus \eqref{e-DimRhoFree} implies $\rho(F+\subspace{x})>\rho(F)$.
Let now $x=x_1+x_2$, where $x_i:=\pi_i(x)$ and $x_2\neq0$. If $x_2\not\in\pi_2(F)$, then \eqref{e-DimRhoFree} implies $\rho(F+\subspace{x})>\rho(F)$.
Thus let $x_2\in\pi_2(F)$. 
Then there exists $\hat{x}\in F$ such that $\pi_2(\hat{x})=x_2$. Hence $\tilde{x}=x-\hat{x}\in E_1\setminus(F\cap E_1)$ and 
$\rho(F+\subspace{x})=\rho(F+\subspace{\tilde{x}})>\rho(F)$ thanks to the previous case.
All of this shows that $F\in\cF(\cM)$.
\\
(b) \eqref{e-DimRhoTriv} is clear from \cref{R-MiPrime}. Let us turn to the identity for $\cF(\cM)$.
\\
``$\supseteq$'' Let $F=F_1\oplus E_2$ for some $F_1\in\cF(\cM_1)$ and let $x\in E\setminus F$. Then $\pi_1(x)\not\in F_1$ and thus $\rho(F+\subspace{x})=\rho_1(\pi_1(F+\subspace{x})=\rho_1(F_1+\subspace{\pi_1(x)})>\rho_1(F_1)=\rho(F)$.
\\
``$\subseteq$'' Let $F\in\cF(\cM)$ and set $F_1=\pi_1(F)$. Then $F\leq F_1\oplus E_2$.
Suppose first that $F\lneq F_1\oplus E_2$.
Then there exists $x_1\in F_1$ and $x_2\in E_2$ such that $x=x_1+x_2\not\in F$. 
Since $F_1=\pi_1(F)$ there exists $\tilde{x}=x_1+\tilde{x}_2\in F$ with $\tilde{x}_2\in E_2$.
Now $\rho(F)<\rho(F+\subspace{x})=\rho(F+\subspace{x-\tilde{x}})=\rho(F+\subspace{x_2-\tilde{x}_2})
=\rho_1(\pi_1(F+\subspace{x_2-\tilde{x}_2}))=\rho_1(F_1)=\rho(F)$, a contradiction. 
Thus $F=F_1\oplus E_2$. 
It remains to show that $F_1\in\cF(\cM_1)$. But this is clear because for any $x_1\in E_1\setminus F_1$ we have 
$x_1\not\in F$ and thus $\rho_1(F_1)=\rho(F)<\rho(F+\subspace{x_1})=\rho_1(F_1+\subspace{x_1})$.
\hfill$\square$

\end{document}